\documentclass[12pt,a4paper,reqno]{article}
\usepackage{amssymb,amsmath,amsfonts}
\usepackage{fullpage}
\usepackage[cp1250]{inputenc}
\usepackage{ifthen}
\usepackage{amssymb}
\usepackage{amsmath}
\usepackage{amsfonts}
\usepackage{latexsym}
\usepackage{cite}
\usepackage{hyperref}
\usepackage{graphics, setspace}
\newcommand{\mathsym}[1]{{}}
\newcommand{\unicode}[1]{{}}

\newtheorem{theorem}{Theorem}[section]

\newtheorem{corollary}{Corollary}[section]

\newtheorem{lemma}{Lemma}[section]

\newtheorem{remark}[theorem]{Remark}

\numberwithin{equation}{section}
\newenvironment{proof}[1][Proof]{\noindent\textbf{#1.} }{\ \rule{0.5em}{0.5em}}
\newenvironment{proof-1}[1][Proof of Theorem \ref{main}]{\noindent\textbf{#1.} }{\ \rule{0.5em}{0.5em}}
\begin{document}

\title{\textbf{\large Radii of the }$\beta -$\textbf{{\large uniformly 
\textbf{convex} of order $\alpha $ of Lommel and Struve functions}}}
\author{{\textbf{\normalsize Sercan Topkaya}}, {\textbf{\normalsize Erhan
Deniz} and}\textbf{\ }{\textbf{{\normalsize \ Murat Ça\u{g}lar}}}}
\date{{\normalsize Department of Mathematics, Faculty of Science and Letters,%
}\\
Kafkas University, Campus, 36100, Kars-Turkey\\
[0.4mm] \textit{topkaya.sercan@hotmail.com}, \textit{edeniz36@gmail.com }\&
\textit{mcaglar25@gmail.com}}
\maketitle

\begin{abstract}
In this paper, we determine the radii of $\beta -$uniformly convex of order $%
\alpha $ for six kinds of normalized Lommel and Struve functions of the
first kind. In the cases considered the normalized Lommel and Struve
functions are $\beta -$uniformly convex functions of order $\alpha $ on the
determined disks.  \\[0.2cm]
{\footnotesize \textbf{Keywords:} Lommel functions, Struve functions,
Univalent functions, }$\beta -${\footnotesize uniformly convex functions of
order }$\alpha ${\footnotesize , Zeros of Lommel functions of the first
kind, Zeros of Struve functions of the first kind. \newline
\textbf{AMS Class:} 30C45, 30C80, 33C10.}
\end{abstract}

\section{Introduction and Preliminaries}

It is well known that the concepts of convexity, starlikeness,
close-to-convexity and uniform convexity including necessary and sufficient
conditions, have a long history as a part of geometric function theory. It
is known that special functions, like Bessel, Struve and Lommel functions of
the first kind have some beautiful geometric properties. Recently, the above
geometric properties of the Bessel functions were investigated by some
earlier results (see \cite{Deniz,Baricz02,Baricz03,Baricz04,Baricz05,Baricz06}). On the other hand the radii of convexity
and starlikeness of the Struve and Lommel functions were studied by Baricz
and his coauthors \cite{Baricz3,Baricz5}. Motivated by the above
developments in this topic, in this paper our aim is to give some new
results for the radius of $\beta -$uniformly convex functions of order $\alpha $ of the normalized Struve and Lommel functions of the first
kind. In the special cases of the parameters $\alpha $ and $\beta $ we can
obtain some earlier results.The key tools in their proofs were some Mittag-Leffler expansions of Lommel and Struve functions of the first kind, special properties of the zeros of these functions and their derivatives.\\
Let $U(z_{0},r)=\{z\in \mathbb{C}:|z-z_{0}|<r\}$ denote the disk
of radius $r$ and center $z_{0}.$ We let $U(r)=U(0,r)$ and $U=U(0,1)=\{z\in 
\mathbb{C}:|z|<1\}.$ Let $(a_{n})_{n\geq 2}$ be a sequence of complex
numbers with 
\begin{equation*}
d=\limsup_{n\rightarrow \infty }|a_{n}|^{\frac{1}{n}}\geq 0,\ \text{and}\
r_{f}=\frac{1}{d}.
\end{equation*}%
If $d=0$ then $r_{f}=+\infty .$ As usual, with $\mathcal{A}$ we denote the
class of analytic functions $f:U(r_{f})\rightarrow 
%TCIMACRO{\U{2102} }%
%BeginExpansion
\mathbb{C}
%EndExpansion
$ of the form 
\begin{equation}
f(z)=z+\sum_{n=2}^{\infty }{a}_{n}z^{n}.  \label{1}
\end{equation}%
We say that a function $f$ of the form (\ref{1}) is convex if $f$ is
univalent and $f(U(r))$ is a convex domain in $\mathbb{C}$. An analytic
description of this definition is 
\begin{equation*}
f\in \mathcal{A}\ \text{is convex\ if and only if}\ \Re \left( 1+\frac{%
zf^{\prime \prime }(z)}{f^{\prime }(z)}\right) >0,\ \ z\in {U(r)}.
\end{equation*}%
The radius of convexity of the function $f$ is defined by 
\begin{equation*}
r_{f}^{c}=\sup \left\{ r\in (0,r_{f}):\ \Re \left( 1+\frac{zf^{\prime \prime
}(z)}{f^{\prime }(z)}\right) >0,\ \ z\in {U}(r)\right\} .
\end{equation*}

In the following we deal with the class of the uniformly convex functions.
Goodman in \cite{Goo} introduced the concept of uniform convexity for
functions of the form (\ref{1}). A function $f$ is said to be uniformly
convex in $U(r)$ if $f$ is of the form (\ref{1}), it is convex, and has the
property that for every circular arc $\gamma $ contained in $U(r)$, with
center $\varsigma $, also in $U(r),$ the arc $f(\gamma )$ is convex. In 1993, R\o nning \cite{10} determined necessary and sufficient conditions
of analytic functions to be uniformly convex in the open unit disk, while in
2002 Ravichandran \cite{Rav} also presented simpler criteria for uniform
convexity. R\o %
nning in \cite{10} give an analytic description of the uniformly convex
functions in the following theorem.

\begin{theorem}
\label{T1} Let $f$ be a function of the form $f(z)=z+\sum_{n=2}^{\infty
}a_{n}z^{n}$ in the disk $U(r).$ The function $f$ is uniformly convex in the
disk $U(r)$ if and only if 
\begin{equation}
\Re\left( 1+\frac{zf^{\prime \prime }(z)}{f^{\prime }(z)}\right) >\left\vert 
\frac{zf^{\prime \prime }(z)}{f^{\prime }(z)}\right\vert ,\ \ \ z\in {U(r).}
\label{d1}
\end{equation}
\end{theorem}

The class of uniformly convex function denote by $UC$. The radius of uniform convexity is defined by 
\begin{equation*}
r_{f}^{uc}=\sup \left\{ r\in (0,r_{f}):\ \Re \left( 1+\frac{zf^{\prime
\prime }(z)}{f^{\prime }(z)}\right) >\left\vert \frac{zf^{\prime \prime }(z)%
}{f^{\prime }(z)}\right\vert ,\ \ z\in {U}(r)\right\} .
\end{equation*}

In 1997, Bharti, Parvatham and Swaminathan defined $\beta
- $uniformly convex functions of order $\alpha$ which is subclass of uniformly convex functions. A function $f\in \mathcal{A}$ is said to be in the class of $\beta
- $uniformly convex functions of order $\alpha ,$ denoted by $\beta
-UC(\alpha )$, if 
\begin{equation}
\Re \left\{ 1+\frac{z{f}^{\prime \prime }(z)}{{f}^{\prime }(z)}\right\}
>\beta \left\vert \frac{z{f}^{\prime \prime }(z)}{{f}^{\prime }(z)}%
\right\vert +\alpha ,  \label{eq2}
\end{equation}%
where $\beta \geqslant 0,\;\alpha \in \lbrack 0,1)$ (see \cite{Bhar}).

This class generalize various other classes which are worthy to mention
here. For example, the class $\beta -UC(0)=\beta -UC$ is the class of $\beta -$uniformly\
convex functions \cite{Kanas1} (also see \cite{Kanas2} and \cite{Kanas3})
and $1-UC(0)=UC$ is the class of uniformly\ convex functions defined by
Goodman \cite{Goo} and Ronning \cite{10}, respectively.

\textbf{Geometric Interpretation. }It is known that $f\in \beta -UC(\alpha )$
if and only if $1+\frac{z{f}^{\prime \prime }(z)}{{f}^{\prime }(z)},$
respectively, takes all the values in the conic domain $\mathcal{R}_{\beta
,\alpha }$ which is included in the right half plane given by%
\begin{equation}
\mathcal{R}_{\beta ,\alpha }:=\left\{ {w=u+iv\in \mathbb{C}:u>\beta \sqrt{%
(u-1)^{2}+v^{2}}+\alpha ,\;\beta \geqslant 0\;}\text{and}{\;\alpha \in \left[
0{,1}\right) }\right\} .  \label{eq4}
\end{equation}%
Denote by $\mathcal{P}(P_{\beta ,\alpha }),$ $\left( \beta {\geq
0,\;0\leqslant \alpha <1}\right) $ the family of functions $p,$ such that $%
p\in \mathcal{P},$ where $\mathcal{P}$ denotes the well-known class of
Caratheodory functions and $p\prec P_{\beta ,\alpha }$ in $U.$ The function $%
P_{\beta ,\alpha }$ maps the unit disk conformally onto the domain $\mathcal{%
R}_{\beta ,\alpha }$ such that $1\in \mathcal{R}_{\beta ,\alpha }$ and $%
\partial \mathcal{R}_{\beta ,\alpha }$ is a curve defined by the equality { \small
\begin{equation}
\partial \mathcal{R}_{\beta ,\alpha }:=\left\{ {w=u+iv\in \mathbb{C}%
:u^{2}=\left( \beta {\sqrt{(u-1)^{2}+v^{2}}+\alpha }\right) ^{2},\;\beta
\geqslant 0\;}\text{and}{\;\alpha \in \left[ 0{,1}\right) }\right\} .
\label{eq5}
\end{equation} }
From elementary computations we see that (\ref{eq5}) represents conic
sections symmetric about the real axis. Thus $\mathcal{R}_{\beta ,\alpha }$
is an elliptic domain for $\beta >1,$ a parabolic domain for $\beta =1,$ a
hyperbolic domain for $0<\beta <1$ and the right half plane $u>\alpha $, for 
$\beta =0.$

The radius of $\beta -$uniform convexity of order $\alpha $ is defined by  { \small
\begin{equation*}  
r_{f}^{\beta -uc(\alpha )}=\sup \left\{ r\in (0,r_{f}):\ \Re \left\{ 1+\frac{%
z{f}^{\prime \prime }(z)}{{f}^{\prime }(z)}\right\} >\beta \left\vert \frac{z%
{f}^{\prime \prime }(z)}{{f}^{\prime }(z)}\right\vert +\alpha ,\text{ }{%
\beta \geqslant 0,\;\alpha \in \left[ 0{,1}\right) ,}\ \ z\in {U}(r)\right\}
.
\end{equation*} }

Our main aim to determine the radii of $\beta -$uniform convexity of order $%
\alpha $ of Lommel and Struve functions.

In order to prove the main results, we need the following lemma given in 
\cite{Deniz}.

\begin{lemma}
\label{L1} i. If $a>b>r\geq |z|,$ and $\lambda \in \lbrack 0,1],$ then 
\begin{equation}
\left\vert \frac{z}{b-z}-\lambda \frac{z}{a-z}\right\vert \leq \frac{r}{b-r}%
-\lambda \frac{r}{a-r}.  \label{2}
\end{equation}
The followings are very simple consequences of this inequality 
\begin{equation}
\Re\left( \frac{z}{b-z}-\lambda \frac{z}{a-z}\right) \leq \frac{r}{b-r}%
-\lambda \frac{r}{a-r}  \label{3}
\end{equation}%
and 
\begin{equation}
\Re\left( \frac{z}{b-z}\right) \leq \left\vert \frac{z}{b-z}\right\vert \leq 
\frac{r}{b-r}.  \label{4}
\end{equation}%
ii. If $b>a>r\geq |z|,$ then%
\begin{equation}
\left\vert \frac{1}{(a+z)(b-z)}\right\vert \leq \frac{1}{(a-r)(b+r)}.
\label{41}
\end{equation}
\end{lemma}

\section{Main results}

In this paper our aim is to consider two classical special functions, the
Lommel function of the first kind $s_{\mu ,\nu }$ and the Struve function of
the first kind $\mathbf{H}_{\nu }.$ They are explicitly defined in terms of
the hypergeometric function $_{1}F_{2}$ by 
\begin{equation*}
s_{\mu ,\nu }(z)=\frac{z^{\mu +1}}{\left( \mu -\nu +1\right) \left( \mu +\nu
+1\right) }\text{ }_{1}F_{2}\left( 1;\frac{\left( \mu -\nu +3\right) }{2},%
\frac{\left( \mu +\nu +3\right) }{2};-\frac{z^{2}}{4}\right) ,
\end{equation*}%
$\frac{1}{2}\left( -\mu \pm \nu -3\right) \notin 
%TCIMACRO{\U{2115} }%
%BeginExpansion
\mathbb{N}
%EndExpansion
$ and 
\begin{equation*}
\mathbf{H}_{\nu }(z)=\frac{\left( \frac{z}{2}\right) ^{\nu +1}}{\sqrt{\frac{%
\pi }{4}}\Gamma \left( \nu +\frac{3}{2}\right) }\text{ }_{1}F_{2}\left( 1;%
\frac{3}{2},\nu +\frac{3}{2};-\frac{z^{2}}{4}\right) ,\text{ \ \ }-\nu -%
\frac{3}{2}\notin 
%TCIMACRO{\U{2115} }%
%BeginExpansion
\mathbb{N}
%EndExpansion
.
\end{equation*}%
Observe that 
\begin{equation*}
s_{\nu ,\nu }(z)=2^{\nu -1}\sqrt{\pi }\Gamma \left( \nu +\frac{1}{2}\right) 
\mathbf{H}_{\nu }(z).
\end{equation*}%
A common feature of these functions is that they are solutions of
inhomogeneous Bessel differantial equations \cite{watson}. Indeed, the
Lommel functions of the first kind $s_{\mu ,\nu }$ is a solution of 
\begin{equation*}
z^{2}w^{\prime \prime }\left( z\right) +zw^{\prime }\left( z\right) +\left(
z^{2}-\nu ^{2}\right) w\left( z\right) =z^{\mu +1}
\end{equation*}%
while the Struve function $\mathbf{H}_{\nu }$ obeys 
\begin{equation*}
z^{2}w^{\prime \prime }\left( z\right) +zw^{\prime }\left( z\right) +\left(
z^{2}-\nu ^{2}\right) w\left( z\right) =\frac{4\left( \frac{z}{2}\right)
^{\nu +1}}{\sqrt{\pi }\Gamma \left( \nu +\frac{1}{2}\right) }.
\end{equation*}%
We refer to Watson's treatise \cite{watson} for comprehensive information
about these functions and recall here briefly some contributions. In $1970$
Steinig \cite{Steinig70} investigated the real zeros of the Struve function $%
\mathbf{H}_{\nu },$ while in 1972 he \cite{Steinig72} examined the sign of $%
s_{\mu ,\nu }(z)$ for real $\mu ,\nu $ and positive $z.$ He showed, among
other things, that for $\mu <\frac{1}{2}$ the function $s_{\mu ,\nu }$ has
infinitely many changes of sign on $(0,\infty ).$ In $2012$ Koumandos and
Lamprecht \cite{Koumandos} obtained sharp estimates for the location of the
zeros of $s_{\mu -\frac{1}{2},\frac{1}{2}}$ when $\mu \in \left( 0,1\right)
. $ Turán type inequalities for $s_{\mu -\frac{1}{2},\frac{1}{2}}$ were
establised in \cite{Baricz1} while those for Struve function were proved in 
\cite{Baricz4}. Geometric properties of the Lommel function $s_{\mu -\frac{1%
}{2},\frac{1}{2}}$ and of the Struve function $\mathbf{H}_{\nu }$ were
obtained in \cite{Baricz2,Baricz3,Baricz5,Orhan,yagmur}.
Motivated by those results, in this paper we are interested on the radii of $%
\beta -$uniformly convex of order $\alpha $ of certain analytic functions
related to the classical special functions under discussion. Since neither $%
s_{\mu -\frac{1}{2},\frac{1}{2}},$ nor $\mathbf{H}_{\nu }$ belongs to the
class analytic functio\i ns, first we perform some natural normalizations,
as in \cite{Baricz3}. We define three functions related to $s_{\mu -\frac{1}{%
2},\frac{1}{2}}:$ 
\begin{equation*}
f_{\mu }(z)=f_{\mu -\frac{1}{2},\frac{1}{2}}(z)=\left( \mu \left( \mu
+1\right) s_{\mu -\frac{1}{2},\frac{1}{2}}(z)\right) ^{\frac{1}{\mu +\frac{1%
}{2}}},
\end{equation*}%
\begin{equation*}
g_{\mu }(z)=g_{\mu -\frac{1}{2},\frac{1}{2}}(z)=\mu \left( \mu +1\right)
z^{-\mu +\frac{1}{2}}s_{\mu -\frac{1}{2},\frac{1}{2}}(z),
\end{equation*}%
and

\begin{equation*}
h_{\mu }(z)=h_{\mu -\frac{1}{2},\frac{1}{2}}(z)=\mu \left( \mu +1\right) z^{%
\frac{3-2\mu }{4}}s_{\mu -\frac{1}{2},\frac{1}{2}}(\sqrt{z}).
\end{equation*}

Similarly, we associate with $\mathbf{H}_{\nu}$ the functions 
\begin{equation*}
u_{\nu}\left( z\right) =\left( \sqrt{\pi }2^{\nu}\text{ }\Gamma \left( \nu +%
\frac{3}{2}\right) \mathbf{H}_{\nu}\left( z\right) \right) ^{\frac{1}{\nu+1}%
},
\end{equation*}

\begin{equation*}
v_{\nu }\left( z\right) =\sqrt{\pi }2^{\nu }z^{-\nu }\text{ }\Gamma \left(
\nu +\frac{3}{2}\right) \mathbf{H}_{\nu }\left( z\right) ,
\end{equation*}%
and%
\begin{equation*}
w_{\nu }\left( z\right) =\sqrt{\pi }2^{\nu }z^{\frac{1-\nu }{2}}\text{ }%
\Gamma \left( \nu +\frac{3}{2}\right) \mathbf{H}_{\nu }\left( \sqrt{z}%
\right) .
\end{equation*}%
Clearly the functions $f_{\mu },$ $g_{\mu },$ $h_{\mu },$ $u_{\nu },$ $%
v_{\nu }$ and $w_{\nu }$ belong to the class of analytic functions $\mathcal{%
A}$. The main results in the present paper concern some interlacing
properties of the zeros of Lommel and Struve functions and derivatives, as
well as the exact values of the radii of $\beta -$uniform convexity of order 
$\alpha $ for these six functions, for some ranges of the parameters.

In the following, we present some lemmas given by Baricz and Ya\u gmur \cite{Baricz5}, on the zeros of derivatives of Lommel an Struve
functions of first kind. These lemmas is one of the key
tools in the proof of our main results.

\begin{lemma}
\label{lm21} The zeros of the Lommel function $s_{\mu -\frac{1}{2},\frac{1}{2%
}}$ and its derivative interlace when $\mu \in \left(-1,1\right),~\mu \neq
0. $ Moreover, the zeros $\xi_{\mu ,n}^{\prime}$ of the function $s_{\mu -%
\frac{1}{2},\frac{1}{2}}^{\prime}$ are all real and simple when $\mu \in
\left(-1,1\right),~\mu \neq 0.$
\end{lemma}

\begin{lemma}
\label{lm22} The zeros of the function $\mathbf{H}_{\nu}$ and its derivative
interlace when $\left|\nu\right|\leq\frac{1}{2}.$ Moreover, the zeros $%
h_{\nu ,n}^{\prime}$ of the function $\mathbf{H}_{\nu}^{\prime}$ are all
real and simple when $\left|\nu\right|\leq\frac{1}{2}.$
\end{lemma}

Now, the first main result of this section presents the radii of $\beta -$%
uniform convexity of order $\alpha $ of functions $f_{\mu },g_{\mu }\text{
and }h_{\mu }.$

\begin{theorem}
\label{thrm21} Let $\mu \in \left( -1,1\right) ,~\mu \neq 0$ and suppose
that $\mu \neq -\frac{1}{2}.$ Then the radius of $\beta -$uniform convexity
of order $\alpha $ of the function $f_{\mu }$ is the smallest positive root
of the equation 
\begin{equation*}
(1-\alpha)+(1+\beta )\left( \frac{rs_{\mu -\frac{1}{2},\frac{1}{2}}^{\prime \prime
}(r)}{s_{\mu -\frac{1}{2},\frac{1}{2}}^{\prime }(r)}+\left( \frac{1}{\mu +%
\frac{1}{2}}-1\right) \frac{rs_{\mu -\frac{1}{2},\frac{1}{2}}^{\prime }(r)}{%
s_{\mu -\frac{1}{2},\frac{1}{2}}(r)}\right) =0.
\end{equation*}%
Moreover $r_{f_{\mu }}^{\beta -uc(\alpha )}<r_{f_{\mu }}^{c}<\xi _{\mu
,1}^{\prime }<\xi _{\mu ,1},$ where $\xi _{\mu ,1}$ and $\xi _{\mu
,1}^{\prime }$ denote the first positive zeros of $s_{\mu -\frac{1}{2},\frac{%
1}{2}}$ and $s_{\mu -\frac{1}{2},\frac{1}{2}}^{\prime },$ respectively and $%
r_{f_{\mu }}^{c}$ is the radius of convexity of the function $f_{\mu }.$
\end{theorem}

\begin{proof}
In \cite{Baricz5}, authors proved the Mittag-Leffler expansions of $s_{\mu -\frac{1}{2},\frac{1}{2}}(z)$ and $s_{\mu -\frac{1}{2},\frac{1}{%
2}}^{\prime }(z)$ as follows:
\begin{equation}\label{yeq1}
s_{\mu -\frac{1}{2},\frac{1}{2}}(z)=\frac{z^{\mu+\frac{1}{2}}}{\mu (\mu+1)}\prod_{n\geq 1}\left(1-\frac{z^2}{\xi _{\mu ,n}^2}\right)
\end{equation}
 and 
\begin{equation}\label{yeq2}s_{\mu -\frac{1}{2},\frac{1}{%
2}}^{\prime }(z)= \frac{\left(\mu+\frac{1}{2}\right)z^{\mu-\frac{1}{2}}}{\mu (\mu+1)}\prod_{n\geq 1}\left(1-\frac{z^2}{\xi _{\mu ,n}^{\prime 2}}\right)\end{equation} 
where $\xi _{\mu ,n}$ and $\xi _{\mu ,n}^{\prime }$ denote the $n$-th positive
roots of $s_{\mu -\frac{1}{2},\frac{1}{2}}$ and $s_{\mu -\frac{1}{2},\frac{1}{%
2}}^{\prime },$ respectively.
Observe also that 
\begin{equation}\label{yeq3} 
1+\frac{zf_{\mu }^{\prime \prime }(z)}{f_{\mu }^{\prime }(z)}=1+\frac{zs_{\mu -\frac{1}{2},\frac{1}{2}}^{\prime \prime
}(z)}{s_{\mu -\frac{1}{2},\frac{1}{2}}^{\prime }(z)}+\left( \frac{1}{\mu +%
\frac{1}{2}}-1\right) \frac{zs_{\mu -\frac{1}{2},\frac{1}{2}}^{\prime }(z)}{%
s_{\mu -\frac{1}{2},\frac{1}{2}}(z)}.
\end{equation} 
Thus, from (\ref{yeq1}), (\ref{yeq2}) and (\ref{yeq3}), we have
\begin{equation*}
1+\frac{zf_{\mu }^{\prime \prime }(z)}{f_{\mu }^{\prime }(z)}=1-\left( \frac{%
1}{\mu +\frac{1}{2}}-1\right) \sum_{n\geq 1}\frac{2z^{2}}{\xi _{\mu
,n}^{2}-z^{2}}-\sum_{n\geq 1}\frac{2z^{2}}{\xi _{\mu ,n}^{\prime 2}-z^{2}}.
\end{equation*}%
Now, the proof will be presented in three cases by considering the intervals
of $\mu $.\newline
Firstly, suppose that $\mu \in \left( 0,\frac{1}{2}\right] .$ Since $\frac{1%
}{\mu +\frac{1}{2}}-1\geq 0,$ inequality (\ref{4}) implies for $\left\vert
z\right\vert \leq r<\xi _{\mu ,1}^{\prime }<\xi _{\mu ,1}$ 
\begin{eqnarray}
\Re \left( 1+\frac{zf_{\mu }^{\prime \prime }(z)}{f_{\mu }^{\prime }(z)}%
\right) &=&1-\sum_{n\geq 1}\Re \left( \frac{2z^{2}}{\xi _{\mu ,n}^{\prime
2}-z^{2}}\right) -\left( \frac{1}{\mu +\frac{1}{2}}-1\right) \sum_{n\geq
1}\Re \left( \frac{2z^{2}}{\xi _{\mu ,n}^{2}-z^{2}}\right)  \label{eq21} \\
&\geq &1-\sum_{n\geq 1}\frac{2r^{2}}{\xi _{\mu ,n}^{\prime 2}-r^{2}}-\left( 
\frac{1}{\mu +\frac{1}{2}}-1\right) \sum_{n\geq 1}\frac{2r^{2}}{\xi _{\mu
,n}^{2}-r^{2}}  \notag \\
&=&1+\frac{rf_{\mu }^{\prime \prime }(r)}{f_{\mu }^{\prime }(r)}.  \notag
\end{eqnarray}%
On the other hand, if in the second part of inequality (\ref{4}) we replace $%
z$ by $z^{2}$ and $b$ by $\xi _{\mu ,n}^{\prime }$ and $\xi _{\mu ,n}$,
respectively, then it follows that 
\begin{equation*}
\left\vert \frac{2z^{2}}{\xi _{\mu ,n}^{\prime 2}-z^{2}}\right\vert \leq 
\frac{2r^{2}}{\xi _{\mu ,n}^{\prime 2}-r^{2}}\text{ and }\left\vert \frac{%
2z^{2}}{\xi _{\mu ,n}^{2}-z^{2}}\right\vert \leq \frac{2r^{2}}{\xi _{\mu
,n}^{2}-r^{2}}
\end{equation*}%
provided that $\left\vert z\right\vert \leq r<\xi _{\mu ,1}^{\prime }<\xi
_{\mu ,1}.$ These two inequalities and the conditions $\frac{1}{\mu +\frac{1%
}{2}}-1\geq 0$ and $\beta \geq 0$, imply that

\begin{eqnarray}
\beta \left\vert \frac{zf_{\mu }^{\prime \prime }(z)}{f_{\mu }^{\prime }(z)}%
\right\vert &=&\beta \left\vert \sum_{n\geq 1}\left( \frac{2z^{2}}{\xi _{\mu
,n}^{\prime 2}-z^{2}}+\left( \frac{1}{\mu +\frac{1}{2}}-1\right) \frac{2z^{2}%
}{\xi _{\mu ,n}^{2}-z^{2}}\right) \right\vert  \label{eq22} \\
&\leq &\beta \sum_{n\geq 1}\left\vert \frac{2z^{2}}{\xi _{\mu ,n}^{\prime
2}-z^{2}}\right\vert +\beta \left( \frac{1}{\mu +\frac{1}{2}}-1\right)
\sum_{n\geq 1}\left\vert \frac{2z^{2}}{\xi _{\mu ,n}^{2}-z^{2}}\right\vert 
\notag \\
&\leq &\beta \sum_{n\geq 1}\left( \frac{2r^{2}}{\xi _{\mu ,n}^{\prime
2}-r^{2}}+\left( \frac{1}{\mu +\frac{1}{2}}-1\right) \frac{2r^{2}}{\xi _{\mu
,n}^{2}-r^{2}}\right) =-\beta \frac{rf_{\mu }^{\prime \prime }(r)}{f_{\mu
}^{\prime }(r)}.  \notag
\end{eqnarray}%
From (\ref{eq21}) and (\ref{eq22}) we infer 
\begin{equation}
\Re \left( 1+\frac{zf_{\mu }^{\prime \prime }(z)}{f_{\mu }^{\prime }(z)}%
\right) -\beta \left\vert \frac{zf_{\mu }^{\prime \prime }(z)}{f_{\mu
}^{\prime }(z)}\right\vert -\alpha \geq 1-\alpha +(1+\beta )\frac{rf_{\mu
}^{\prime \prime }(r)}{f_{\mu }^{\prime }(r)},  \label{eq23}
\end{equation}
where $\left\vert z\right\vert \leq
r<\xi _{\mu ,1}^{\prime }\text{ and }{\alpha \in \left[ 0{,1}\right) ,}\text{
}\beta \geq 0. $ \\
In the second step we will prove that inequalities (\ref{eq21}) and (\ref%
{eq22}) hold in the case $\mu \in \left( \frac{1}{2},1\right),$ too. Indeed
in the case $\mu \in \left( \frac{1}{2},1\right) $ the roots $0<\xi _{\mu
,n}^{\prime }<\xi _{\mu ,n}$ are real for every natural number $n$.
Moreover, inequality (\ref{4}) implies that 
\begin{equation*}
\Re \left( \frac{2z^{2}}{\xi _{\mu ,n}^{\prime 2}-z^{2}}\right) \leq
\left\vert \frac{2z^{2}}{\xi _{\mu ,n}^{\prime 2}-z^{2}}\right\vert \leq 
\frac{2r^{2}}{\xi _{\mu ,n}^{\prime 2}-r^{2}},~~\left\vert z\right\vert \leq
r<\xi _{\mu ,1}^{\prime }<\xi _{\mu ,1}
\end{equation*}%
and 
\begin{equation*}
\Re \left( \frac{2z^{2}}{\xi _{\mu ,n}^{2}-z^{2}}\right) \leq \left\vert 
\frac{2z^{2}}{\xi _{\mu ,n}^{2}-z^{2}}\right\vert \leq \frac{2r^{2}}{\xi
_{\mu ,n}^{2}-r^{2}},~~\left\vert z\right\vert \leq r<\xi _{\mu ,1}^{\prime
}<\xi _{\mu ,1}.
\end{equation*}%
Putting $\lambda =1-\frac{1}{\mu +\frac{1}{2}}$ inequality (\ref{3}) implies 
\begin{equation*}
\Re \left( \frac{2z^{2}}{\xi _{\mu ,n}^{\prime 2}-z^{2}}-\left( 1-\frac{1}{%
\mu +\frac{1}{2}}\right) \frac{2z^{2}}{\xi _{\mu ,n}^{2}-z^{2}}\right) \leq 
\frac{2r^{2}}{\xi _{\mu ,n}^{\prime 2}-r^{2}}-\left( 1-\frac{1}{\mu +\frac{1%
}{2}}\right) \frac{2r^{2}}{\xi _{\mu ,n}^{2}-r^{2}},
\end{equation*}%
for $\left\vert z\right\vert \leq r<\xi _{\mu ,1}^{\prime }<\xi _{\mu ,1},$
and we get 
\begin{eqnarray}
\Re \left( 1+\frac{zf_{\mu }^{\prime \prime }(z)}{f_{\mu }^{\prime }(z)}%
\right) &=&1-\sum_{n\geq 1}\Re \left( \frac{2z^{2}}{\xi _{\mu ,n}^{\prime
2}-z^{2}}-\left( 1-\frac{1}{\mu +\frac{1}{2}}\right) \frac{2z^{2}}{\xi _{\mu
,n}^{2}-z^{2}}\right)  \label{eq24} \\
&\geq &1-\sum_{n\geq 1}\left( \frac{2r^{2}}{\xi _{\mu ,n}^{\prime 2}-r^{2}}%
-\left( 1-\frac{1}{\mu +\frac{1}{2}}\right) \frac{2r^{2}}{\xi _{\mu
,n}^{2}-r^{2}}\right)  \notag \\
&=&1+\frac{rf_{\mu }^{\prime \prime }(r)}{f_{\mu }^{\prime }(r)}.  \notag
\end{eqnarray}%
Now, if in the inequality (\ref{2}) we replace $z$ by $z^{2}$ and $b$ by $%
\xi _{\mu ,n}^{\prime }$ and $\xi _{\mu ,n}$ we again put $\lambda =1-\frac{1%
}{\mu +\frac{1}{2}},$ it follows that 
\begin{equation*}
\left\vert \frac{2z^{2}}{\xi _{\mu ,n}^{\prime 2}-z^{2}}-\left( 1-\frac{1}{%
\mu +\frac{1}{2}}\right) \frac{2z^{2}}{\xi _{\mu ,n}^{2}-z^{2}}\right\vert
\leq \frac{2r^{2}}{\xi _{\mu ,n}^{\prime 2}-r^{2}}-\left( 1-\frac{1}{\mu +%
\frac{1}{2}}\right) \frac{2r^{2}}{\xi _{\mu ,n}^{2}-r^{2}},
\end{equation*}%
provided that $\left\vert z\right\vert \leq r<\xi _{\mu ,1}^{\prime }<\xi
_{\mu ,1}.$ Thus, for $\beta \geq 0$ we obtain 
\begin{eqnarray}
\beta \left\vert \frac{zf_{\mu }^{\prime \prime }(z)}{f_{\mu }^{\prime }(z)}%
\right\vert &=&\beta \left\vert \sum_{n\geq 1}\left( \frac{2z^{2}}{\xi _{\mu
,n}^{\prime 2}-z^{2}}-\left( 1-\frac{1}{\mu +\frac{1}{2}}\right) \frac{2z^{2}%
}{\xi _{\mu ,n}^{2}-z^{2}}\right) \right\vert  \label{eq25} \\
&\leq &\beta \sum_{n\geq 1}\left\vert \frac{2z^{2}}{\xi _{\mu ,n}^{\prime
2}-z^{2}}-\left( 1-\frac{1}{\mu +\frac{1}{2}}\right) \frac{2z^{2}}{\xi _{\mu
,n}^{2}-z^{2}}\right\vert  \notag \\
&\leq &\beta \sum_{n\geq 1}\left( \frac{2r^{2}}{\xi _{\mu ,n}^{\prime
2}-r^{2}}-\left( 1-\frac{1}{\mu +\frac{1}{2}}\right) \frac{2r^{2}}{\xi _{\mu
,n}^{2}-r^{2}}\right) =-\beta \frac{rf_{\mu }^{\prime \prime }(r)}{f_{\mu
}^{\prime }(r)}.  \notag
\end{eqnarray}%
Finally the following inequality be infered from (\ref{eq24}) and (\ref{eq25}%
) for $\mu \in \left( \frac{1}{2},1\right) $ 
\begin{equation}
\Re \left( 1+\frac{zf_{\mu }^{\prime \prime }(z)}{f_{\mu }^{\prime }(z)}%
\right) -\beta \left\vert \frac{zf_{\mu }^{\prime \prime }(z)}{f_{\mu
}^{\prime }(z)}\right\vert -\alpha \geq 1-\alpha +(1+\beta )\frac{rf_{\mu
}^{\prime \prime }(r)}{f_{\mu }^{\prime }(r)},~~  \label{eq26}
\end{equation}
where $\left\vert z\right\vert \leq
r<\xi _{\mu ,1}^{\prime }\text{ and }{\alpha \in \left[ 0{,1}\right) ,}\text{
}\beta \geq 0. $ \\
Equality holds (\ref{eq24}) if and only if $z=r.$ Thus it follows that 
\begin{equation*}
\inf_{\left\vert z\right\vert <r}\left[ \Re \left( 1+\frac{zf_{\mu }^{\prime
\prime }(z)}{f_{\mu }^{\prime }(z)}\right) -\beta \left\vert \frac{zf_{\mu
}^{\prime \prime }(z)}{f_{\mu }^{\prime }(z)}\right\vert -\alpha \right]
=1-\alpha +(1+\beta )\frac{rf_{\mu }^{\prime \prime }(r)}{f_{\mu }^{\prime
}(r)},
\end{equation*}%
for $r\in (0,\xi _{\mu ,1}^{\prime }),~{\alpha \in \left[ 0{,1}%
\right) ,}\text{ }\beta \geq 0 \text{ and } \mu \in (0,1).$ \\
The mapping $\psi _{\mu }:(0,\xi _{\mu ,1}^{\prime })\rightarrow \mathbb{R}$
defined by 
\begin{equation*}
\psi _{\mu }(r)=1-\alpha +(1+\beta )\frac{rf_{\mu }^{\prime \prime }(r)}{%
f_{\mu }^{\prime }(r)}=1-\alpha -(1+\beta )\sum_{n\geq 1}\left( \frac{2r^{2}%
}{\xi _{\mu ,n}^{\prime 2}-r^{2}}-\left( 1-\frac{1}{\mu +\frac{1}{2}}\right) 
\frac{2r^{2}}{\xi _{\mu ,n}^{2}-r^{2}}\right)
\end{equation*}%
is strictly decreasing for all $\mu \in (0,1)$ and ${\alpha \in \left[ 0{,1}%
\right) ,}$ $\beta \geq 0$. Namely, we obtain 
\begin{eqnarray*}
\psi _{\mu }^{\prime }(r) &=&-(1+\beta )\sum_{n\geq 1}\left( \frac{4r^{2}\xi
_{\mu ,n}^{\prime 2}}{\left( \xi _{\mu ,n}^{\prime 2}-r^{2}\right) ^{2}}%
-\left( 1-\frac{1}{\mu +\frac{1}{2}}\right) \frac{4r^{2}\xi _{\mu ,n}^{2}}{%
\left( \xi _{\mu ,n}^{2}-r^{2}\right) ^{2}}\right) \\
&<&(1+\beta )\sum_{n\geq 1}\left( \frac{4r^{2}\xi _{\mu ,n}^{2}}{\left( \xi
_{\mu ,n}^{2}-r^{2}\right) ^{2}}-\frac{4r^{2}\xi _{\mu ,n}^{\prime 2}}{%
\left( \xi _{\mu ,n}^{\prime 2}-r^{2}\right) ^{2}}\right) <0
\end{eqnarray*}%
for $\mu \in (\frac{1}{2},1),~r\in (0,\xi _{\mu ,1}^{\prime })\text{ and }%
\beta \geq 0.$ Here we used again that the zeros $\xi _{\mu ,n}$ and $\xi
_{\mu ,n}^{\prime }$ interlace, and for all $n\in \mathbb{N},~\mu \in (0,1)$
and $r<\sqrt{\xi _{\mu ,n}\xi _{\mu ,n}^{\prime }}$ we have that 
\begin{equation*}
\xi _{\mu ,n}^{2}\left( \xi _{\mu ,n}^{\prime 2}-r^{2}\right) ^{2}<\xi _{\mu
,n}^{\prime 2}\left( \xi _{\mu ,n}^{2}-r^{2}\right) ^{2}.
\end{equation*}%
Let now $\mu \in \left( 0,\frac{1}{2}\right] $ and $r>0.$ Thus the following
inequality%
\begin{equation*}
\psi _{\mu }^{\prime }(r)=-(1+\beta )\sum_{n\geq 1}\left( \frac{4r^{2}\xi
_{\mu ,n}^{\prime 2}}{\left( \xi _{\mu ,n}^{\prime 2}-r^{2}\right) ^{2}}%
-\left( 1-\frac{1}{\mu +\frac{1}{2}}\right) \frac{4r^{2}\xi _{\mu ,n}^{2}}{%
\left( \xi _{\mu ,n}^{2}-r^{2}\right) ^{2}}\right) <0
\end{equation*}
is satisfy and thus $\psi _{\mu }$ is indeed strictly decreasing for all $%
\mu \in (0,1)$ and $\beta \geq 0$.\newline
Now, since $\lim_{r\searrow 0}\psi _{\mu }(r)=1$ and $\lim_{r\nearrow \xi
_{\mu ,1}^{\prime }}\psi _{\mu }(r)=-\infty $, in wiew of the minimum
principle for harmonic functions it follows that for $\mu \in (0,1)$ and $%
z\in U(r_{f_{\mu }}^{\beta -uc(\alpha )})$ we have 
\begin{equation}
\Re \left( 1+\frac{zf_{\mu }^{\prime \prime }(z)}{f_{\mu }^{\prime }(z)}%
\right) -\beta \left\vert \frac{zf_{\mu }^{\prime \prime }(z)}{f_{\mu
}^{\prime }(z)}\right\vert -\alpha >0  \label{eq27}
\end{equation}%
if and only if $r_{f_{\mu }}^{\beta -uc(\alpha )}$ is the unique root of 
\begin{equation}
1+(1+\beta )\frac{rf_{\mu }^{\prime \prime }(r)}{f_{\mu }^{\prime }(r)}%
=\alpha ,~~\alpha \in \lbrack 0,1)\text{ and }\beta \geq 0.  \label{eq271}
\end{equation}

In the final step we will proved that inequality (\ref{eq27}) also holds
when $\mu \in (-1,0).$\newline
In order to this, suppose that $\mu \in (0,1)$ and since (see \cite{Baricz1}%
) the function $z\rightarrow \varphi _{0}(z)=\mu (\mu +1)z^{-\mu -\frac{1}{2}%
}s_{\mu -\frac{1}{2},\frac{1}{2}}(z)$ belongs to the Laguerre-P\' olya
class of entire functions, it satisfies the Laguerre inequality 
\begin{equation*}
\left( \varphi _{0}^{(n)}(z)\right) ^{2}-\left( \varphi
_{0}^{(n-1)}(z)\right) \left( \varphi _{0}^{(n+1)}(z)\right) >0,
\end{equation*}%
where $\mu \in (0,1)$ and $z\in \mathbb{R}$.\newline
Substituting $\mu $ by $\mu -1$, $\varphi _{0}$ by the function $\varphi
_{1}(z)=\mu (\mu -1)z^{-\mu +\frac{1}{2}}s_{\mu -\frac{3}{2},\frac{1}{2}}(z)$
and taking into account that the $n$th positive zeros of $\varphi _{1}$ and $%
\varphi _{1}^{\prime },$ denoted by $\zeta _{\mu ,n}$ and $\zeta _{\mu
,n}^{\prime }$, interlace, since $\varphi _{1}$ belongs also to the
Laguerre-P\' olya class of entire functions (see \cite{Baricz1}). It is
worth mentioning that

\begin{equation}
\Re \left( 1+\frac{zf_{\mu -1}^{\prime \prime }(z)}{f_{\mu -1}^{\prime }(z)}%
\right) \geq 1-\sum_{n\geq 1}\frac{2r^{2}}{\zeta _{\mu ,n}^{\prime 2}-r^{2}}%
-\left( \frac{1}{\mu -\frac{1}{2}}-1\right) \sum_{n\geq 1}\frac{2r^{2}}{%
\zeta _{\mu ,n}^{2}-r^{2}}  \label{eq28}
\end{equation}%
and 
\begin{equation}
\beta \left\vert \frac{zf_{\mu -1}^{\prime \prime }(z)}{f_{\mu -1}^{\prime
}(z)}\right\vert \leq \beta \sum_{n\geq 1}\left( \frac{2r^{2}}{\zeta _{\mu
,n}^{\prime 2}-r^{2}}+\left( \frac{1}{\mu -\frac{1}{2}}-1\right) \frac{2r^{2}%
}{\zeta _{\mu ,n}^{2}-r^{2}}\right)  \label{eq29}
\end{equation}%
holds for $\mu \in (0,1),~\beta \geq 0$ and $\mu \neq \frac{1}{2}$. In this
case we use again minimum principle for harmonic functions to ensure that (%
\ref{eq27}) is valid for $\mu -1$ instead of $\mu .$ Consequently, replacing $%
\mu $ by $\mu +1,$ the equation (\ref{eq271}) is satisfy for $\mu \in
(-1,0), $ $\alpha \in \lbrack 0,1)$ and $\beta \geq 0.$ Thus the proof is
complete.
\end{proof}

As a result of the Theorem \ref{thrm21}, the following corollary is obtained by taking $\alpha=0$ ve $\beta=1$.

\begin{corollary}\label{snc21}
Let $\mu \in \left( -1,1\right) ,~\mu \neq 0$ and suppose
that $\mu \neq -\frac{1}{2}.$ Then the radius of uniform convexity
of the function $f_{\mu }$ is the smallest positive root
of the equation 
\begin{equation*}
1+2\left( \frac{rs_{\mu -\frac{1}{2},\frac{1}{2}}^{\prime \prime
}(r)}{s_{\mu -\frac{1}{2},\frac{1}{2}}^{\prime }(r)}+\left( \frac{1}{\mu +%
\frac{1}{2}}-1\right) \frac{rs_{\mu -\frac{1}{2},\frac{1}{2}}^{\prime }(r)}{%
s_{\mu -\frac{1}{2},\frac{1}{2}}(r)}\right) =0.
\end{equation*}%
Moreover $r_{f_{\mu }}^{uc}<r_{f_{\mu }}^{c}<\xi _{\mu
,1}^{\prime }<\xi _{\mu ,1},$ where $\xi _{\mu ,1}$ and $\xi _{\mu
,1}^{\prime }$ denote the first positive zeros of $s_{\mu -\frac{1}{2},\frac{%
1}{2}}$ and $s_{\mu -\frac{1}{2},\frac{1}{2}}^{\prime },$ respectively and $%
r_{f_{\mu }}^{c}$ is the radius of convexity of the function $f_{\mu }.$
\end{corollary}

\begin{center}
\includegraphics{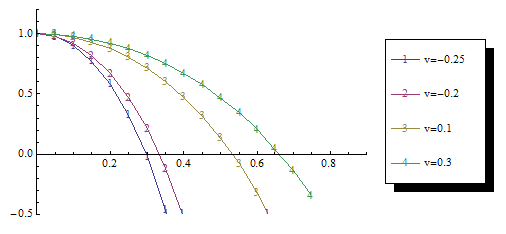}

The graph of the fuction $r\mapsto 1+2\left( \frac{rs_{\mu -\frac{1}{2},\frac{1}{2}}^{\prime \prime
}(r)}{s_{\mu -\frac{1}{2},\frac{1}{2}}^{\prime }(r)}+\left( \frac{1}{\mu +%
\frac{1}{2}}-1\right) \frac{rs_{\mu -\frac{1}{2},\frac{1}{2}}^{\prime }(r)}{%
s_{\mu -\frac{1}{2},\frac{1}{2}}(r)}\right) $ \vskip 0.3 cm

for $\mu \in \{-0.25, -0.2, 0.1, 0.3\}$ on $[0,0.9]$
\end{center}

\begin{theorem}
\label{thrm22} Let $\mu \in \left( -1,1\right) ,~\mu \neq 0$ and suppose
that $\mu \neq -\frac{1}{2}.$ Then the radius of $\beta -$uniform convexity
of order $\alpha $ of the function $g_{\mu }$ is the smallest positive root
of the equation 
\begin{equation*}
(1-\alpha)-(1+\beta )\left( \frac{1}{2}+\mu -r\frac{(\frac{3}{^{2}}-\mu )s_{\mu -%
\frac{1}{2},\frac{1}{2}}^{\prime }(r)+rs_{\mu -\frac{1}{2},\frac{1}{2}%
}^{\prime \prime }(r)}{(\frac{1}{^{2}}-\mu )s_{\mu -\frac{1}{2},\frac{1}{2}%
}(r)+rs_{\mu -\frac{1}{2},\frac{1}{2}}^{\prime }(r)}\right) =0.
\end{equation*}%
Moreover $r_{g_{\mu }}^{\beta -uc(\alpha )}<r_{g_{\mu }}^{c}<\gamma _{\mu
,1}^{\prime }<\xi _{\mu ,1},$ where $\xi _{\mu ,1}$ and $\gamma _{\mu ,1}$
denote the first positive zeros of $s_{\mu -\frac{1}{2},\frac{1}{2}}$ and $%
g_{\mu }^{\prime },$ respectively.
\end{theorem}

\begin{proof}
Let $\xi _{\mu ,n}$ and $\gamma _{\mu ,n}$ denote the $n$-th positive root
of $s_{\mu -\frac{1}{2},\frac{1}{2}}$ and $g_{\mu }^{\prime },$ respectively
and the smallest positive root of $g_{\mu }^{\prime }$ does not exceed the
first positive root of $s_{\mu -\frac{1}{2},\frac{1}{2}}$. In \cite{Baricz5}
with the help of Hadamard's Theorem \cite[p.26]{Levin}, the following
equality was proved: 
\begin{equation*}
1+\frac{zg_{\mu }^{\prime \prime }(z)}{g_{\mu }^{\prime }(z)}=1-\sum_{n\geq
1}\frac{2z^{2}}{\gamma _{\mu ,n}^{2}-z^{2}}.
\end{equation*}%
By using inequality (\ref{4}), for all $z\in U(\gamma _{\mu ,1})$ we have
the inequality 
\begin{equation}
\Re \left( 1+\frac{zg_{\mu }^{\prime \prime }(z)}{g_{\mu }^{\prime }(z)}%
\right) \geq 1-\sum_{n\geq 1}\frac{2r^{2}}{\gamma _{\mu ,n}^{2}-r^{2}}
\label{eq210}
\end{equation}%
where $\left\vert z\right\vert =r.$\newline
On the other hand, again by using inequality (\ref{4}), for all $z\in
U(\gamma _{\mu ,1})\text{ and }\beta \geq 0$ we get the inequality 
\begin{eqnarray}
\beta \left\vert \frac{zg_{\mu }^{\prime \prime }(z)}{g_{\mu }^{\prime }(z)}%
\right\vert &=&\beta \left\vert \sum_{n\geq 1}\frac{2z^{2}}{\gamma _{\mu
,n}^{2}-z^{2}}\right\vert  \label{eq211} \\
&\leq &\beta \sum_{n\geq 1}\left\vert \frac{2z^{2}}{\gamma _{\mu
,n}^{2}-z^{2}}\right\vert  \notag \\
&\leq &\beta \sum_{n\geq 1}\frac{2r^{2}}{\gamma _{\mu ,n}^{2}-r^{2}}=-\beta 
\frac{rg_{\mu }^{\prime \prime }(r)}{g_{\mu }^{\prime }(r)}.  \notag
\end{eqnarray}%
Finally the following inequality be infered from (\ref{eq210}) and (\ref%
{eq211}) 
\begin{equation}
\Re \left( 1+\frac{zg_{\mu }^{\prime \prime }(z)}{g_{\mu }^{\prime }(z)}%
\right) -\beta \left\vert \frac{zg_{\mu }^{\prime \prime }(z)}{g_{\mu
}^{\prime }(z)}\right\vert -\alpha \geq 1-\alpha +(1+\beta )\frac{rg_{\mu
}^{\prime \prime }(r)}{g_{\mu }^{\prime }(r)},~~\beta \geq 0,\alpha \in
\lbrack 0,1),  \label{eq212}
\end{equation}%
where $\left\vert z\right\vert =r.$ Thus, for $r\in (0,\gamma _{\mu
,1}),~\beta \geq 0\text{ and }\alpha \in \lbrack 0,1)$ we obtain 
\begin{equation*}
\inf_{\left\vert z\right\vert <r}\left[ \Re \left( 1+\frac{zg_{\mu }^{\prime
\prime }(z)}{g_{\mu }^{\prime }(z)}\right) -\beta \left\vert \frac{zg_{\mu
}^{\prime \prime }(z)}{g_{\mu }^{\prime }(z)}\right\vert -\alpha \right]
=1-\alpha +(1+\beta )\frac{rg_{\mu }^{\prime \prime }(r)}{g_{\mu }^{\prime
}(r)}.
\end{equation*}%
The mapping $\Theta _{\mu }:(0,\gamma _{\mu ,1})\rightarrow \mathbb{R}$
defined by 
\begin{equation*}
\Theta _{\mu }(r)=1-\alpha +(1+\beta )\frac{rg_{\mu }^{\prime \prime }(r)}{%
g_{\mu }^{\prime }(r)}=1-\alpha -(1+\beta )\sum_{n\geq 1}\frac{2r^{2}}{%
\gamma _{\mu ,n}^{2}-r^{2}}
\end{equation*}%
is strictly decreasing since $\lim_{r\searrow 0}\Theta _{\mu }(r)=1$ and $%
\lim_{r\nearrow \gamma _{\mu ,1}}\Theta _{\mu }(r)=-\infty $. As a result in
wiew of the minimum principle for harmonic functions it follows that for $%
\alpha \in \lbrack 0,1),~\beta \geq 0$ and $z\in U(r_{1})$ we have 
\begin{equation*}
\Re \left( 1+\frac{zg_{\mu }^{\prime \prime }(z)}{g_{\mu }^{\prime }(z)}%
\right) -\beta \left\vert \frac{zg_{\mu }^{\prime \prime }(z)}{g_{\mu
}^{\prime }(z)}\right\vert -\alpha >0
\end{equation*}%
if and only if $r_{1}$ is the unique root of 
\begin{equation*}
1+(1+\beta )\frac{rg_{\mu }^{\prime \prime }(r)}{g_{\mu }^{\prime }(r)}%
=\alpha ,~~\alpha \in \lbrack 0,1)\text{ and }\beta \geq 0
\end{equation*}%
situated in $(0,\gamma _{\mu ,1})$.
\end{proof}

As a result of the Theorem \ref{thrm22}, the next corollary is obtained by taking $\alpha=0$ ve $\beta=1$.

\begin{corollary}\label{snc22}
Let $\mu \in \left( -1,1\right) ,~\mu \neq 0$ and suppose
that $\mu \neq -\frac{1}{2}.$ Then the radius of uniform convexity
of the function $g_{\mu }$ is the smallest positive root
of the equation 
\begin{equation*}
1-2\left( \frac{1}{2}+\mu -r\frac{(\frac{3}{^{2}}-\mu )s_{\mu -%
\frac{1}{2},\frac{1}{2}}^{\prime }(r)+rs_{\mu -\frac{1}{2},\frac{1}{2}%
}^{\prime \prime }(r)}{(\frac{1}{^{2}}-\mu )s_{\mu -\frac{1}{2},\frac{1}{2}%
}(r)+rs_{\mu -\frac{1}{2},\frac{1}{2}}^{\prime }(r)}\right) =0.
\end{equation*}%
Moreover $r_{g_{\mu }}^{uc}<r_{g_{\mu }}^{c}<\gamma _{\mu
,1}^{\prime }<\xi _{\mu ,1},$ where $\xi _{\mu ,1}$ and $\gamma _{\mu ,1}$
denote the first positive zeros of $s_{\mu -\frac{1}{2},\frac{1}{2}}$ and $%
g_{\mu }^{\prime },$ respectively.
\end{corollary}

\begin{center}
\includegraphics{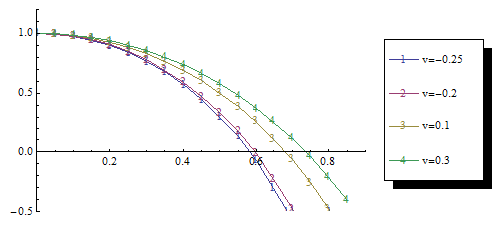}

The graph of the fuction $r\mapsto 1-2\left( \frac{1}{2}+\mu -r\frac{(\frac{3}{^{2}}-\mu )s_{\mu -%
\frac{1}{2},\frac{1}{2}}^{\prime }(r)+rs_{\mu -\frac{1}{2},\frac{1}{2}%
}^{\prime \prime }(r)}{(\frac{1}{^{2}}-\mu )s_{\mu -\frac{1}{2},\frac{1}{2}%
}(r)+rs_{\mu -\frac{1}{2},\frac{1}{2}}^{\prime }(r)}\right) $  \vskip 0.3 cm

for $\mu \in \{-0.25, -0.2, 0.1, 0.3\}$ on $[0,0.9]$
\end{center}

\begin{theorem}
\label{thrm23} Let $\mu \in \left( -1,1\right) ,~\mu \neq 0$ and suppose
that $\mu \neq -\frac{1}{2}.$ Then the radius of $\beta -$uniform convexity
of order $\alpha $ of the function $h_{\mu }$ is the smallest positive root
of the equation 
\begin{equation*}
4(1-\alpha)-(1+\beta )\left( 1+2\mu -2\sqrt{r}\frac{(\frac{5}{^{2}}-\mu )s_{\mu -\frac{%
1}{2},\frac{1}{2}}^{\prime }(\sqrt{r})+\sqrt{r}s_{\mu -\frac{1}{2},\frac{1}{2%
}}^{\prime \prime }(\sqrt{r})}{(\frac{3}{^{2}}-\mu )s_{\mu -\frac{1}{2},%
\frac{1}{2}}(\sqrt{r})+\sqrt{r}s_{\mu -\frac{1}{2},\frac{1}{2}}^{\prime }(%
\sqrt{r})}\right) =0.
\end{equation*}%
Moreover $r_{h_{\mu }}^{\beta -uc(\alpha )}<r_{h_{\mu }}^{c}<\delta _{\mu
,1}^{\prime }<\xi _{\mu ,1},$ where $\xi _{\mu ,1}$ and $\delta _{\mu ,1}$
denote the first positive zeros of $s_{\mu -\frac{1}{2},\frac{1}{2}}$ and $%
h_{\mu }^{\prime },$ respectively.
\end{theorem}

\begin{proof}
Let $\xi _{\mu ,n}$ and $\delta _{\mu ,n}$ denote the $n$-th positive root
of $s_{\mu -\frac{1}{2},\frac{1}{2}}$ and $h_{\mu }^{\prime },$ respectively
and the smallest positive root of $h_{\mu }^{\prime }$ does not exceed the
first positive root of $s_{\mu -\frac{1}{2},\frac{1}{2}}$. In \cite{Baricz5}
with the help of Hadamard's Theorem \cite[p.26]{Levin}, the following
equality was proved: 
\begin{equation*}
1+\frac{zh_{\mu }^{\prime \prime }(z)}{h_{\mu }^{\prime }(z)}=1-\sum_{n\geq
1}\frac{z}{\delta _{\mu ,n}^{2}-z}.
\end{equation*}%
By using inequality (\ref{4}), for all $z\in U(\delta _{\mu ,1})$ we obtain
the inequality 
\begin{equation}
\Re \left( 1+\frac{zh_{\mu }^{\prime \prime }(z)}{h_{\mu }^{\prime }(z)}%
\right) \geq 1-\sum_{n\geq 1}\frac{r}{\delta _{\mu ,n}^{2}-r}  \label{eq213}
\end{equation}%
where $\left\vert z\right\vert =r.$\newline
Moreover, again by using inequality (\ref{4}), for all $z\in U(\delta _{\mu
,1})\text{ and }\beta \geq 0$ we get the inequality 
\begin{eqnarray}
\beta \left\vert \frac{zh_{\mu }^{\prime \prime }(z)}{h_{\mu }^{\prime }(z)}%
\right\vert &=&\beta \left\vert \sum_{n\geq 1}\frac{z}{\delta _{\mu ,n}^{2}-z%
}\right\vert  \label{eq214} \\
&\leq &\beta \sum_{n\geq 1}\left\vert \frac{z}{\delta _{\mu ,n}^{2}-z}%
\right\vert  \notag \\
&\leq &\beta \sum_{n\geq 1}\frac{r}{\delta _{\mu ,n}^{2}-r}=-\beta \frac{%
rh_{\mu }^{\prime \prime }(r)}{h_{\mu }^{\prime }(r)}.  \notag
\end{eqnarray}%
As a result, from (\ref{eq213}) and (\ref{eq214}) we have 
\begin{equation}
\Re \left( 1+\frac{zh_{\mu }^{\prime \prime }(z)}{h_{\mu }^{\prime }(z)}%
\right) -\beta \left\vert \frac{zh_{\mu }^{\prime \prime }(z)}{h_{\mu
}^{\prime }(z)}\right\vert -\alpha \geq 1-\alpha +(1+\beta )\frac{rh_{\mu
}^{\prime \prime }(r)}{h_{\mu }^{\prime }(r)}  \label{eq215}
\end{equation}%
where $\left\vert z\right\vert =r.$ Thus, for $r\in (0,\delta _{\mu
,1}),~\beta \geq 0\text{ and }\alpha \in \lbrack 0,1)$ we have 
\begin{equation*}
\inf_{\left\vert z\right\vert <r}\left[ \Re \left( 1+\frac{zh_{\mu }^{\prime
\prime }(z)}{h_{\mu }^{\prime }(z)}\right) -\beta \left\vert \frac{zh_{\mu
}^{\prime \prime }(z)}{h_{\mu }^{\prime }(z)}\right\vert -\alpha \right]
=1-\alpha +(1+\beta )\frac{rh_{\mu }^{\prime \prime }(r)}{h_{\mu }^{\prime
}(r)}.
\end{equation*}%
The mapping $\Phi _{\mu }:(0,\delta _{\mu ,1})\rightarrow \mathbb{R}$
defined by 
\begin{equation*}
\Phi _{\mu }(r)=1-\alpha +(1+\beta )\frac{rh_{\mu }^{\prime \prime }(r)}{%
h_{\mu }^{\prime }(r)}=1-\alpha -(1+\beta )\sum_{n\geq 1}\frac{r}{\delta
_{\mu ,n}^{2}-r}
\end{equation*}%
is strictly decreasing since $\lim_{r\searrow 0}\Phi _{\mu }(r)=1>\alpha $
and $\lim_{r\nearrow \delta _{\mu ,1}}\Phi _{\mu }(r)=-\infty $.
Consequently, in wiew of the minimum principle for harmonic functions it
follows that for $\alpha \in \lbrack 0,1),~\beta \geq 0$ and $z\in U(r_{2})%
\mathbb{\ }$we have 
\begin{equation*}
\Re \left( 1+\frac{zh_{\mu }^{\prime \prime }(z)}{h_{\mu }^{\prime }(z)}%
\right) -\beta \left\vert \frac{zh_{\mu }^{\prime \prime }(z)}{h_{\mu
}^{\prime }(z)}\right\vert -\alpha >0
\end{equation*}%
if and only if $r_{2}$ is the unique root of 
\begin{equation*}
1+(1+\beta )\frac{rh_{\mu }^{\prime \prime }(r)}{h_{\mu }^{\prime }(r)}%
=\alpha ,~~\alpha \in \lbrack 0,1)\text{ and }\beta \geq 0.
\end{equation*}%
situated in $(0,\delta _{\mu ,1})$.
\end{proof}

As a result of the Theorem \ref{thrm23}, the following corollary is obtained by taking $\alpha=0$ ve $\beta=1$.

\begin{corollary}\label{snc23}
Let $\mu \in \left( -1,1\right) ,~\mu \neq 0$ and suppose
that $\mu \neq -\frac{1}{2}.$ Then the radius of uniform convexity
of the function $h_{\mu }$ is the smallest positive root
of the equation 
\begin{equation*}
\frac{1}{2}-\mu  +\sqrt{r}\frac{(\frac{5}{^{2}}-\mu )s_{\mu -\frac{%
1}{2},\frac{1}{2}}^{\prime }(\sqrt{r})+\sqrt{r}s_{\mu -\frac{1}{2},\frac{1}{2%
}}^{\prime \prime }(\sqrt{r})}{(\frac{3}{^{2}}-\mu )s_{\mu -\frac{1}{2},%
\frac{1}{2}}(\sqrt{r})+\sqrt{r}s_{\mu -\frac{1}{2},\frac{1}{2}}^{\prime }(%
\sqrt{r})}=0.
\end{equation*}%
Moreover $r_{h_{\mu }}^{uc}<r_{h_{\mu }}^{c}<\delta _{\mu
,1}^{\prime }<\xi _{\mu ,1},$ where $\xi _{\mu ,1}$ and $\delta _{\mu ,1}$
denote the first positive zeros of $s_{\mu -\frac{1}{2},\frac{1}{2}}$ and $%
h_{\mu }^{\prime },$ respectively.
\end{corollary}

\begin{center}
\includegraphics{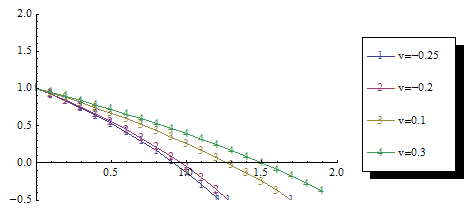}

The graph of the fuction $r\mapsto \frac{1}{2}-\mu  +\sqrt{r}\frac{(\frac{5}{^{2}}-\mu )s_{\mu -\frac{%
1}{2},\frac{1}{2}}^{\prime }(\sqrt{r})+\sqrt{r}s_{\mu -\frac{1}{2},\frac{1}{2%
}}^{\prime \prime }(\sqrt{r})}{(\frac{3}{^{2}}-\mu )s_{\mu -\frac{1}{2},%
\frac{1}{2}}(\sqrt{r})+\sqrt{r}s_{\mu -\frac{1}{2},\frac{1}{2}}^{\prime }(%
\sqrt{r})} $ \vskip 0.3 cm

for $\mu \in \{-0.25, -0.2, 0.1, 0.3\}$ on $[0,2]$
\end{center}

For $\mu=\frac{1}{3}$, Lommel functions defined in terms of the hypergeometric function $ \text{ }_{1}F_{2}$ as follows:
$$s_{-\frac{1}{5},\frac{1}{2}}(z)=\frac{100z^{4/5}}{39} \text{ }_{1}F_{2}\left( 1;\frac{23}{20},\frac{33}{20};-\frac{z^2}{4}  \right). $$
Then, we have 
$$f_{\frac{3}{10}}(z)= z\left[ \text{ }_{1}F_{2}\left( 1;\frac{23}{20},\frac{33}{20};-\frac{z^2}{4} \right)\right]^{5/4},~~ g_{\frac{3}{10}}(z)=z\text{ }_{1}F_{2}\left( 1;\frac{23}{20},\frac{33}{20};-\frac{z^2}{4} \right)  $$
and $$ h_{\frac{3}{10}}(z)=z \text{ }_{1}F_{2}\left( 1;\frac{23}{20},\frac{33}{20};-\frac{z}{4}  \right) . $$ 
We obtain the following results for the functions $f_{\frac{3}{10}},~g_{\frac{3}{10}}$ and $h_{\frac{3}{10}} $.
\begin{itemize}
\item $ f_{\frac{3}{10}}(z) \in UC$ in the disk $U(r_1=0.6623)$,
\item $ g_{\frac{3}{10}}(z) \in UC$ in the disk $U(r_2=0.7376)$,
\item $ h_{\frac{3}{10}}(z) \in UC$ in the disk $U(r_3=1.4961)$,
\end{itemize}
where $r_1,~r_2 $ and $r_3$ is the smallest positive root of the equations given Corollary \ref{snc21}-\ref{snc23} for $\mu=\frac{1}{3}$.

Secondly, the other main result of this section presents the $\beta -$%
uniform convexity of order $\alpha $ of functions $u_{\nu },v_{\nu }\text{
and }w_{\nu },$ related to Struve's one. The first part of next theorem is
an interesting of Lemma \ref{lm22}.

\begin{theorem}
\label{thrm24} Let $\left\vert \nu \right\vert \leq \frac{1}{2}$ and $0\leq
\alpha <1.$ Then the radius of $\beta -$uniform convexity of order $\alpha $
of the function $u_{\nu }$ is the smallest positive root of the equation 
\begin{equation*}
(1-\alpha)+(1+\beta )\left( \frac{r\mathbf{H}_{\nu }^{\prime \prime }(r)}{\mathbf{H}%
_{\nu }^{\prime }(r)}+\left( \frac{1}{\nu +1}-1\right) \frac{r\mathbf{H}%
_{\nu }^{\prime }(r)}{\mathbf{H}_{\nu }(r)}\right) =0 .
\end{equation*}%
Moreover $r_{u_{\nu }}^{\beta -uc(\alpha )}<r_{u_{\nu }}^{c}<h_{\nu
,1}^{\prime }<h_{\nu ,1},$ where $h_{\nu ,1}$ and $h_{\nu ,1}^{\prime }$
denote the first positive zeros of $\mathbf{H}_{\nu }$ and $\mathbf{H}_{\nu
}^{\prime },$ respectively.
\end{theorem}

\begin{proof}
We note that $$ 1+\frac{zu_{\nu }^{\prime \prime }(z)}{u_{\nu }^{\prime }(z)}= 1+\frac{z%
\mathbf{H}_{\nu }^{\prime \prime }(z)}{\mathbf{H}_{\nu }^{\prime }(z)}%
+\left( \frac{ 1}{\nu +1 }-1\right)\frac{z\mathbf{H}_{\nu }^{\prime }(z)}{%
\mathbf{H}_{\nu }(z)}  $$
Using the Mittag-Leffler expansions of $\mathbf{H}_{\nu}$ and $\mathbf{H}_{\nu}^{\prime}$ \cite[Theorem 4]{Baricz5} given by
\begin{equation}\label{yeq4}
\mathbf{H}_{\nu}(z)=\frac{z^{\nu+1}}{\sqrt{\pi}\Gamma\left(\nu+\frac{3}{2}\right)}\prod_{n\geq 1}\left(1-\frac{z^2}{h_{\nu ,n}^2}\right)
\end{equation}
and 
\begin{equation}\label{yeq5}
\mathbf{H}_{\nu}^{\prime}(z)=\frac{\left(\nu+1\right)z^{\nu}}{\sqrt{\pi}2^{\nu}\Gamma\left(\nu+\frac{3}{2}\right)}\prod_{n\geq 1}\left(1-\frac{z^2}{h_{\nu ,n}^{\prime 2}}\right)
\end{equation}
where $h_{\nu ,n}$ and $h_{\nu ,n}^{\prime}$ denote the $n$-th positive root
of $\mathbf{H}_{\nu}$ and $\mathbf{H}_{\nu}^{\prime},$ respectively. From (\ref{yeq4}) and (\ref{yeq4}), we obtain 
$$\frac{z\mathbf{H}_{\nu}^{\prime}(z)}{\mathbf{H}_{\nu}(z)}=\nu+1-\sum_{n\geq 1} \frac{2z^2}{h_{\nu ,n}^2-z^2},~~1+\frac{z\mathbf{H}_{\nu}^{\prime \prime}(z)}{\mathbf{H}_{\nu}^{\prime}(z)}=\nu+1-\sum_{n\geq 1} \frac{2z^2}{h_{\nu ,n}^{\prime 2}-z^2}.$$ 
Thus, we have
\begin{eqnarray*}
1+\frac{zu_{\nu }^{\prime \prime }(z)}{u_{\nu }^{\prime }(z)}= 
1-\left( \frac{ 1}{\nu +1 }-1\right) \sum_{n\geq 1}\frac{2z^{2}}{h_{\nu
,n}^{2}-z^{2}}-\sum_{n\geq 1}\frac{2z^{2}}{h_{\nu ,n}^{\prime 2}-z^{2}}.
\end{eqnarray*}
Now, the proof will be presented in two cases by considering the intervals
of $\nu$.\newline
Firstly, suppose that $\nu \in \left[-\frac{1}{2},0\right].$ Since $\frac{ 1}{%
\nu +1 }-1\geq 0,$ inequality (\ref{4}) implies 
\begin{eqnarray}  \label{eq216}
\Re \left(1+\frac{zu_{\nu }^{\prime \prime }(z)}{u_{\nu }^{\prime }(z)}%
\right)&=& 1-\sum_{n\geq 1}\Re \left(\frac{2z^{2}}{h_{\nu ,n}^{\prime
2}-z^{2}}\right)-\left(\frac{ 1}{\nu +1 }-1\right)\sum_{n\geq 1}\Re \left(%
\frac{2z^{2}}{h_{\nu ,n}^{2}-z^{2}} \right) \\
&\geq& 1-\sum_{n\geq 1}\frac{2r^{2}}{h_{\nu ,n}^{\prime 2}-r^{2}}-\left(%
\frac{ 1}{\nu +1 }-1\right)\sum_{n\geq 1}\frac{2r^{2}}{h_{\nu ,n}^{2}-r^{2}}
\notag \\
&=& 1+\frac{ru_{\nu }^{\prime \prime }(r)}{u_{\nu }^{\prime }(r)}.  \notag
\end{eqnarray}
On the other hand, if in the second part of inequality (\ref{4}) we replace $%
z$ by $z^2$ and by $h_{\nu ,1}^{\prime}$ and $h_{\nu ,1}$, respectively,
then it follows that 
\begin{equation*}
\left|\frac{2z^{2}}{h_{\nu ,n}^{\prime 2}-z^{2}} \right| \leq \frac{2r^{2}}{%
h_{\nu ,n}^{\prime 2}-r^{2}} \text{ and } \left|\frac{2z^{2}}{h_{\nu
,n}^{2}-z^{2}}\right|\leq \frac{2r^{2}}{h_{\nu ,n}^{2}-r^{2}}
\end{equation*}
provided that $\left|z\right|\leq r< h_{\nu ,1}^{\prime}<h_{\nu ,1}.$ These
two inequalities and the conditions $\frac{ 1}{\nu +1 }-1 \geq 0$ and $\beta
\geq 0$, imply that

\begin{eqnarray}
\beta \left\vert \frac{zu_{\nu }^{\prime \prime }(z)}{u_{\nu }^{\prime }(z)}%
\right\vert  &=&\beta \left\vert \sum_{n\geq 1}\left( \frac{2z^{2}}{h_{\nu
,n}^{\prime 2}-z^{2}}+\left( \frac{1}{\nu +1}-1\right) \frac{2z^{2}}{h_{\nu
,n}^{2}-z^{2}}\right) \right\vert   \label{eq217} \\
&\leq &\beta \sum_{n\geq 1}\left\vert \frac{2z^{2}}{h_{\nu ,n}^{\prime
2}-z^{2}}\right\vert +\beta \left( \frac{1}{\nu +1}-1\right) \sum_{n\geq
1}\left\vert \frac{2z^{2}}{h_{\nu ,n}^{2}-z^{2}}\right\vert   \notag \\
&\leq &\beta \sum_{n\geq 1}\left( \frac{2r^{2}}{h_{\nu ,n}^{\prime 2}-r^{2}}%
+\left( \frac{1}{\nu +1}-1\right) \frac{2r^{2}}{h_{\nu ,n}^{2}-r^{2}}\right)
=-\beta \frac{ru_{\nu }^{\prime \prime }(r)}{u_{\nu }^{\prime }(r)}.  \notag
\end{eqnarray}%
From (\ref{eq216}) and (\ref{eq217}) we get 
\begin{equation}
\Re \left( 1+\frac{zu_{\nu }^{\prime \prime }(z)}{u_{\nu }^{\prime }(z)}%
\right) -\beta \left\vert \frac{zu_{\nu }^{\prime \prime }(z)}{u_{\nu
}^{\prime }(z)}\right\vert -\alpha \geq 1-\alpha +(1+\beta )\frac{ru_{\nu
}^{\prime \prime }(r)}{u_{\nu }^{\prime }(r)},\text{ \ \ } \label{eq218}
\end{equation}
where $\left\vert
z\right\vert \leq r<h_{\nu ,1}^{\prime }\text{ and }\beta \geq 0,\alpha \in
\lbrack 0,1). $\\
Secondly, in the case $\nu \in \left[ 0,\frac{1}{2}\right] $ the roots $%
0<h_{\nu ,1}^{\prime }<h_{\nu ,1},$ are real for every natural number $n.$
Moreover, inequality (\ref{4}) implies that 
\begin{equation*}
\Re \left( \frac{2z^{2}}{h_{\nu ,n}^{\prime 2}-z^{2}}\right) \leq \left\vert 
\frac{2z^{2}}{h_{\nu ,n}^{\prime 2}-z^{2}}\right\vert \leq \frac{2r^{2}}{%
h_{\nu ,n}^{\prime 2}-r^{2}},~~\left\vert z\right\vert \leq r<h_{\nu
,1}^{\prime }<h_{\nu ,1}
\end{equation*}%
and 
\begin{equation*}
\Re \left( \frac{2z^{2}}{h_{\nu ,n}^{2}-z^{2}}\right) \leq \left\vert \frac{%
2z^{2}}{h_{\nu ,n}^{2}-z^{2}}\right\vert \leq \frac{2r^{2}}{h_{\nu
,n}^{2}-r^{2}},~~\left\vert z\right\vert \leq r<h_{\nu ,1}^{\prime }<h_{\nu
,1}.
\end{equation*}%
Putting $\lambda =1-\frac{1}{\nu +1},$ inequality (\ref{3}) implies 
\begin{equation*}
\Re \left( \frac{2z^{2}}{h_{\nu ,n}^{\prime 2}-z^{2}}-\left( 1-\frac{1}{\nu
+1}\right) \frac{2z^{2}}{h_{\nu ,n}^{2}-z^{2}}\right) \leq \frac{2r^{2}}{%
h_{\nu ,n}^{\prime 2}-r^{2}}-\left( 1-\frac{1}{\nu +1}\right) \frac{2r^{2}}{%
h_{\nu ,n}^{2}-r^{2}},
\end{equation*}%
for $\left\vert z\right\vert \leq r<h_{\nu ,1}^{\prime }<h_{\nu ,1},$ and we
get 
\begin{eqnarray}
\Re \left( 1+\frac{zu_{\nu }^{\prime \prime }(z)}{u_{\nu }^{\prime }(z)}%
\right)  &=&1-\sum_{n\geq 1}\Re \left( \frac{2z^{2}}{h_{\nu ,n}^{\prime
2}-z^{2}}-\left( 1-\frac{1}{\nu +1}\right) \frac{2z^{2}}{h_{\nu ,n}^{2}-z^{2}%
}\right)   \label{eq219} \\
&\geq &1-\sum_{n\geq 1}\left( \frac{2r^{2}}{h_{\nu ,n}^{\prime 2}-r^{2}}%
-\left( 1-\frac{1}{\nu +1}\right) \frac{2r^{2}}{h_{\nu ,n}^{2}-r^{2}}\right) 
\notag \\
&=&1+\frac{ru_{\nu }^{\prime \prime }(r)}{u_{\nu }^{\prime }(r)}.  \notag
\end{eqnarray}%
Now, if in the inequality (\ref{2}) we replace $z$ by $z^{2}$ and we again
put $\lambda =1-\frac{1}{\nu +1},$ it follows that 
\begin{equation*}
\left\vert \frac{2z^{2}}{h_{\nu ,n}^{\prime 2}-z^{2}}-\left( 1-\frac{1}{\nu
+1}\right) \frac{2z^{2}}{h_{\nu ,n}^{2}-z^{2}}\right\vert \leq \frac{2r^{2}}{%
h_{\nu ,n}^{\prime 2}-r^{2}}-\left( 1-\frac{1}{\nu +1}\right) \frac{2r^{2}}{%
h_{\nu ,n}^{2}-r^{2}},
\end{equation*}%
provided that $\left\vert z\right\vert \leq r<h_{\nu ,1}^{\prime }<h_{\nu
,1}.$ Thus, for $\beta \geq 0$ we have 
\begin{eqnarray}
\beta \left\vert \frac{zu_{\nu }^{\prime \prime }(z)}{u_{\nu }^{\prime }(z)}%
\right\vert  &=&\beta \left\vert \sum_{n\geq 1}\left( \frac{2z^{2}}{h_{\nu
,n}^{\prime 2}-z^{2}}-\left( 1-\frac{1}{\nu +1}\right) \frac{2z^{2}}{h_{\nu
,n}^{2}-z^{2}}\right) \right\vert   \label{eq220} \\
&\leq &\beta \sum_{n\geq 1}\left\vert \frac{2z^{2}}{h_{\nu ,n}^{\prime
2}-z^{2}}-\left( 1-\frac{1}{\nu +1}\right) \frac{2z^{2}}{h_{\nu ,n}^{2}-z^{2}%
}\right\vert   \notag \\
&\leq &\beta \sum_{n\geq 1}\left( \frac{2r^{2}}{h_{\nu ,n}^{\prime 2}-r^{2}}%
-\left( 1-\frac{1}{\nu +1}\right) \frac{2r^{2}}{h_{\nu ,n}^{2}-r^{2}}\right)
=-\beta \frac{ru_{\nu }^{\prime \prime }(r)}{u_{\nu }^{\prime }(r)}.  \notag
\end{eqnarray}%
As a result, the following inequality be infered from (\ref{eq219}) and (\ref%
{eq220}) such as (\ref{eq216}) and (\ref{eq217}) 
\begin{equation}
\Re \left( 1+\frac{zu_{\nu }^{\prime \prime }(z)}{u_{\nu }^{\prime }(z)}%
\right) -\beta \left\vert \frac{zu_{\nu }^{\prime \prime }(z)}{u_{\nu
}^{\prime }(z)}\right\vert -\alpha \geq 1-\alpha +(1+\beta )\frac{ru_{\nu
}^{\prime \prime }(r)}{u_{\nu }^{\prime }(r)},~~
\label{eq221}
\end{equation}
where $\left\vert z\right\vert \leq
r<h_{\nu ,1}^{\prime }\text{ and }\beta \geq 0,\alpha \in \lbrack 0,1). $ \\
Equality holds (\ref{eq218}) and (\ref{eq221}) if and only if $z=r.$ Thus it
follows that 
\begin{equation*}
\inf_{\left\vert z\right\vert <r}\left[ \Re \left( 1+\frac{zu_{\nu }^{\prime
\prime }(z)}{u_{\nu }^{\prime }(z)}\right) -\beta \left\vert \frac{zu_{\nu
}^{\prime \prime }(z)}{u_{\nu }^{\prime }(z)}\right\vert -\alpha \right]
=1-\alpha +(1+\beta )\frac{ru_{\nu }^{\prime \prime }(r)}{u_{\nu }^{\prime
}(r)},~~
\end{equation*}
where $r\in (0,h_{\nu ,1}^{\prime })\text{ and }\beta \geq 0,\alpha \in
\lbrack 0,1). $ \\
The mapping $\psi _{\nu }:(0,h_{\nu ,1}^{\prime })\rightarrow \mathbb{R}$
defined by 
\begin{equation*}
\psi _{\nu }(r)=1+(1+\beta )\frac{ru_{\nu }^{\prime \prime }(r)}{u_{\nu
}^{\prime }(r)}=1-(1+\beta )\sum_{n\geq 1}\left( \frac{2r^{2}}{h_{\nu
,n}^{\prime 2}-r^{2}}-\left( 1-\frac{1}{\nu +1}\right) \frac{2r^{2}}{h_{\nu
,n}^{2}-r^{2}}\right) 
\end{equation*}%
is strictly decreasing for all $\left\vert \nu \right\vert \leq \frac{1}{2}$
and $\beta \geq 0$. Namely, we obtain 
\begin{eqnarray*}
\psi _{\nu }^{\prime }(r) &=&-(1+\beta )\sum_{n\geq 1}\left( \frac{4rh_{\nu
,n}^{\prime 2}}{\left( h_{\nu ,n}^{\prime 2}-r^{2}\right) ^{2}}-\left( 1-%
\frac{1}{\nu +1}\right) \frac{4rh_{\nu ,n}^{2}}{\left( h_{\nu
,n}^{2}-r^{2}\right) ^{2}}\right)  \\
&<&(1+\beta )\sum_{n\geq 1}\left( \frac{4rh_{\nu ,n}^{2}}{\left( h_{\nu
,n}^{2}-r^{2}\right) ^{2}}-\frac{4rh_{\nu ,n}^{\prime 2}}{\left( h_{\nu
,n}^{\prime 2}-r^{2}\right) ^{2}}\right) <0
\end{eqnarray*}%
for $\nu \in \left[ 0,\frac{1}{2}\right] ,~r\in (0,h_{\nu ,1}^{\prime })%
\text{ and }\beta \geq 0.$ Here we used again that the zeros $h_{\nu ,n}$
and $h_{\nu ,n}^{\prime }$ interlace, and for all $n\in \mathbb{N}%
,~\left\vert \nu \right\vert \leq \frac{1}{2}$ and $r<\sqrt{h_{\nu ,n}h_{\nu
,n}^{\prime }}$ we have that 
\begin{equation*}
h_{\nu ,n}^{2}\left( h_{\nu ,n}^{\prime 2}-r^{2}\right) ^{2}<h_{\nu
,n}^{\prime 2}\left( h_{\nu ,n}^{2}-r^{2}\right) ^{2}.
\end{equation*}%
Observe that when $\nu \in \left[ \frac{1}{2},0\right] $ and $r>0$ we have
also that $\psi _{\nu }^{\prime }(r)<0$, and thus $\psi _{\nu }$ is indeed
strictly decreasing for all $\left\vert \nu \right\vert \leq \frac{1}{2}$
and $\beta \geq 0$.\newline
Now, since $\lim_{r\searrow 0}\psi _{\nu }(r)=1$ and $\lim_{r\nearrow h_{\nu
,1}^{\prime }}\psi _{\nu }(r)=-\infty $, in wiew of the minimum principle
for harmonic functions it follows that for $\left\vert \nu \right\vert \leq 
\frac{1}{2}$ and $z\in U({r_{3}})$ we get 
\begin{equation}
\Re \left( 1+\frac{zu_{\nu }^{\prime \prime }(z)}{u_{\nu }^{\prime }(z)}%
\right) -\beta \left\vert \frac{zu_{\nu }^{\prime \prime }(z)}{u_{\nu
}^{\prime }(z)}\right\vert >\alpha .  \label{eq222}
\end{equation}%
if and only if $r_{3}$ is the unique root of 
\begin{equation*}
1+(1+\beta )\frac{ru_{\nu }^{\prime \prime }(r)}{u_{\nu }^{\prime }(r)}%
=\alpha ,~~\alpha \in \lbrack 0,1)\text{ and }\beta \geq 0.
\end{equation*}%
situated in $\left( 0,h_{\nu ,1}^{\prime }\right) .$
\end{proof}

As a result of the Theorem \ref{thrm24}, the next corollary is obtained by taking $\alpha=0$ ve $\beta=1$.

\begin{corollary}\label{snc24}
Let $\left\vert \nu \right\vert \leq \frac{1}{2}$. Then the radius of uniform convexity of the function $u_{\nu }$ is the smallest positive root of the equation 
\begin{equation*}
1+2\left( \frac{r\mathbf{H}_{\nu }^{\prime \prime }(r)}{\mathbf{H}%
_{\nu }^{\prime }(r)}+\left( \frac{1}{\nu +1}-1\right) \frac{r\mathbf{H}%
_{\nu }^{\prime }(r)}{\mathbf{H}_{\nu }(r)}\right) =0.
\end{equation*}%
Moreover $r_{u_{\nu }}^{uc}<r_{u_{\nu }}^{c}<h_{\nu
,1}^{\prime }<h_{\nu ,1},$ where $h_{\nu ,1}$ and $h_{\nu ,1}^{\prime }$
denote the first positive zeros of $\mathbf{H}_{\nu }$ and $\mathbf{H}_{\nu
}^{\prime },$ respectively.
\end{corollary}

\begin{center}
\includegraphics{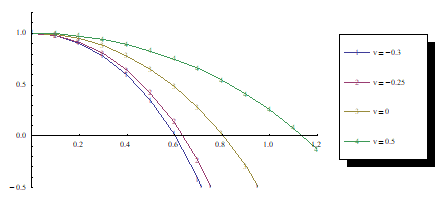}

The graph of the fuction $r\mapsto 1+2\left( \frac{r\mathbf{H}_{\nu }^{\prime \prime }(r)}{\mathbf{H}%
_{\nu }^{\prime }(r)}+\left( \frac{1}{\nu +1}-1\right) \frac{r\mathbf{H}%
_{\nu }^{\prime }(r)}{\mathbf{H}_{\nu }(r)}\right) $  \vskip 0.3 cm

for $\nu \in \{-0.3, -0.25, 0, 0.5\}$ on $[0,1.2]$
\end{center}

\begin{theorem}
\label{thrm25} Let $\left\vert \nu \right\vert \leq \frac{1}{2}$ and $0\leq
\alpha <1.$ Then the radius of $\beta -$uniform convexity of order $\alpha $
of the function $v_{\nu }$ is the smallest positive root of the equation 
\begin{equation*}
(1-\alpha)-\left( 1+\beta \right) \left( 1+\nu -r\frac{(1-\nu )\mathbf{H}_{\nu
}^{\prime }(r)+r\mathbf{H}_{\nu }^{\prime \prime }(r)}{-\nu \mathbf{H}_{\nu
}(r)+r\mathbf{H}_{\nu }^{\prime }(r)}\right) =0.
\end{equation*}%
Moreover $r_{v _{\nu }}^{\beta -uc(\alpha )}<r_{v _{\nu }}^{c}<\varsigma
_{\nu ,1}<h_{\nu ,1},$ where $h_{\nu ,1}$ and $\varsigma _{\nu ,1}$ denote
the first positive zeros of $\mathbf{H}_{\nu }$ and $v_{\nu }^{\prime },$
respectively.
\end{theorem}

\begin{proof}
Let $h_{\nu ,n}$ and $\varsigma _{\nu ,n}$ denote the $n$-th positive root
of $\mathbf{H}_{\nu }$ and $v_{\nu }^{\prime },$ respectively and the
smallest positive root of $v_{\nu }^{\prime }$ does not exceed the first
positive root of $\mathbf{H}_{\nu }$. In \cite{Baricz5}, the following equality was proved: 
\begin{equation*}
1+\frac{zv_{\nu }^{\prime \prime }(z)}{v_{\nu }^{\prime }(z)}=-\nu +z\frac{%
(1-\nu )\mathbf{H}_{\nu }^{\prime }(z)+z\mathbf{H}_{\nu }^{\prime \prime }(z)%
}{-\nu \mathbf{H}_{\nu }(z)+z\mathbf{H}_{\nu }^{\prime }(z)}=1-\sum_{n\geq 1}%
\frac{2z^{2}}{\varsigma _{\nu ,n}^{2}-z^{2}}.
\end{equation*}%
By using inequality (\ref{4}), for all $z\in U(\varsigma _{\nu ,1})$ we have
the inequality 
\begin{equation}
\Re \left( 1+\frac{zv_{\nu }^{\prime \prime }(z)}{v_{\nu }^{\prime }(z)}%
\right) \geq 1-\sum_{n\geq 1}\frac{2r^{2}}{\varsigma _{\nu ,n}^{2}-r^{2}}
\label{eq223}
\end{equation}%
where $\left\vert z\right\vert =r.$\newline
On the other hand, again by using inequality (\ref{4}), for all $z\in
U(\varsigma _{\nu ,1})\text{ and }\beta \geq 0$ we get the inequality 
\begin{eqnarray}
\beta \left\vert \frac{zv_{\nu }^{\prime \prime }(z)}{v_{\nu }^{\prime }(z)}%
\right\vert  &=&\beta \left\vert \sum_{n\geq 1}\frac{2z^{2}}{\varsigma _{\nu
,n}^{2}-z^{2}}\right\vert   \label{eq224} \\
&\leq &\beta \sum_{n\geq 1}\left\vert \frac{2z^{2}}{\varsigma _{\nu
,n}^{2}-z^{2}}\right\vert   \notag \\
&\leq &\beta \sum_{n\geq 1}\frac{2r^{2}}{\varsigma _{\nu ,n}^{2}-r^{2}}%
=-\beta \frac{rv_{\nu }^{\prime \prime }(r)}{v_{\nu }^{\prime }(r)}.  \notag
\end{eqnarray}%
Finally the following inequality be infered from (\ref{eq223}) and (\ref%
{eq224}) 
\begin{equation}
\Re \left( 1+\frac{zv_{\nu }^{\prime \prime }(z)}{v_{\nu }^{\prime }(z)}%
\right) -\beta \left\vert \frac{zv_{\nu }^{\prime \prime }(z)}{v_{\nu
}^{\prime }(z)}\right\vert -\alpha \geq 1-\alpha +(1+\beta )\frac{rv_{\nu
}^{\prime \prime }(r)}{v_{\nu }^{\prime }(r)},~~\beta \geq 0, \label{eq225}
\end{equation}%
where $\left\vert z\right\vert =r.$ Thus, for $r\in (0,\varsigma _{\nu
,1}),~\beta \geq 0\text{ and }\alpha \in \lbrack 0,1)$ we obtain 
\begin{equation*}
\inf_{\left\vert z\right\vert <r}\left[ \Re \left( 1+\frac{zv_{\nu }^{\prime
\prime }(z)}{v_{\nu }^{\prime }(z)}\right) -\beta \left\vert \frac{zv_{\nu
}^{\prime \prime }(z)}{v_{\nu }^{\prime }(z)}\right\vert -\alpha \right]
=1-\alpha +(1+\beta )\frac{rv_{\nu }^{\prime \prime }(r)}{v_{\nu }^{\prime
}(r)}.
\end{equation*}%
The mapping $\Theta _{\nu }:(0,\varsigma _{\nu ,1})\rightarrow \mathbb{R}$
defined by 
\begin{equation*}
\Theta _{\nu }(r)=1+(1+\beta )\frac{rv_{\nu }^{\prime \prime }(r)}{v_{\nu
}^{\prime }(r)}=1-(1+\beta )\sum_{n\geq 1}\frac{2r^{2}}{\varsigma _{\nu
,n}^{2}-r^{2}}
\end{equation*}%
is strictly decreasing since $\lim_{r\searrow 0}\Theta _{\nu }(r)=1$ and $%
\lim_{r\nearrow \varsigma _{\nu ,1}}\Theta _{\nu }(r)=-\infty $. As a result
in wiew of the minimum principle for harmonic functions it follows that for $%
\alpha \in \lbrack 0,1),~\beta \geq 0$ and $z\in U(r_{4})$ we have 
\begin{equation*}
\Re \left( 1+\frac{zv_{\nu }^{\prime \prime }(z)}{v_{\nu }^{\prime }(z)}%
\right) -\beta \left\vert \frac{zv_{\nu }^{\prime \prime }(z)}{v_{\nu
}^{\prime }(z)}\right\vert >\alpha .
\end{equation*}%
if and only if $r_{4}$ is the unique root of 
\begin{equation*}
1+(1+\beta )\frac{rv_{\nu }^{\prime \prime }(r)}{v_{\nu }^{\prime }(r)}%
=\alpha ,~~\alpha \in \lbrack 0,1)\text{ and }\beta \geq 0,
\end{equation*}%
situated in $(0,\varsigma _{\nu ,1}).$
\end{proof}

As a result of the Theorem \ref{thrm25}, the following corollary is obtained by taking $\alpha=0$ ve $\beta=1$.

\begin{corollary}\label{snc25}
Let $\left\vert \nu \right\vert \leq \frac{1}{2}$. Then the radius of uniform convexity of the function $v_{\nu }$ is the smallest positive root of the equation 
\begin{equation*}
1-2\left( 1+\nu -r\frac{(1-\nu )\mathbf{H}_{\nu
}^{\prime }(r)+r\mathbf{H}_{\nu }^{\prime \prime }(r)}{-\nu \mathbf{H}_{\nu
}(r)+r\mathbf{H}_{\nu }^{\prime }(r)}\right) =0.
\end{equation*}%
Moreover $r_{v _{\nu }}^{uc}<r_{v _{\nu }}^{c}<\varsigma
_{\nu ,1}<h_{\nu ,1},$ where $h_{\nu ,1}$ and $\varsigma _{\nu ,1}$ denote
the first positive zeros of $\mathbf{H}_{\nu }$ and $v_{\nu }^{\prime },$
respectively.
\end{corollary}

\begin{center}
\includegraphics{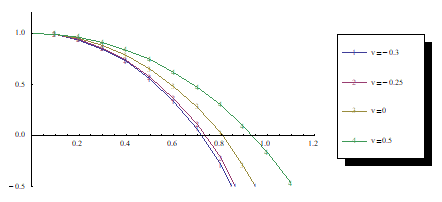}

The graph of the fuction $r\mapsto 1-2 \left( 1+\nu -r\frac{(1-\nu )\mathbf{H}_{\nu
}^{\prime }(r)+r\mathbf{H}_{\nu }^{\prime \prime }(r)}{-\nu \mathbf{H}_{\nu
}(r)+r\mathbf{H}_{\nu }^{\prime }(r)}\right) $  \vskip 0.3 cm

for $\nu \in \{-0.3, -0.25, 0, 0.5\}$ on $[0,1.2]$
\end{center}

\begin{theorem}
\label{thrm26} Let $\left\vert \nu \right\vert \leq \frac{1}{2}$ and $0\leq
\alpha <1.$ Then the radius of $\beta-$uniformly convex of order $\alpha $ of the function $w_{\nu }$ is the smallest positive root of the equation 
\begin{equation*}
2(1-\alpha)-\left( 1+\beta \right) \left( 1+\nu -\sqrt{r}\frac{(2-\nu )\mathbf{H}_{\nu
}^{\prime }(\sqrt{r})+\sqrt{r}\mathbf{H}_{\nu }^{\prime \prime }(\sqrt{r})}{%
(1-\nu )\mathbf{H}_{\nu }(\sqrt{r})+\sqrt{r}\mathbf{H}_{\nu }^{\prime }(%
\sqrt{r})}\right) =0 .
\end{equation*}%
Moreover $r_{w_{\nu }}^{\beta -uc(\alpha )}<r_{w_{\nu }}^{c}<\sigma _{\nu
,1}<h_{\nu ,1},$ where $h_{\nu ,1}$ and $\sigma _{\nu ,1}$ denote the first
positive zeros of $\mathbf{H}_{\nu }$ and $w_{\nu }^{\prime },$ respectively.
\end{theorem}

\begin{proof}
Let $h_{\nu ,n}$ and $\sigma _{\nu ,n}$ denote the $n$-th positive root of $%
\mathbf{H}_{\nu }$ and $w_{\nu }^{\prime },$ respectively and the smallest
positive root of $w_{\nu }^{\prime }$ does not exceed the first positive
root of $\mathbf{H}_{\nu }$. In \cite{Baricz5}, the following equality was proved,
\begin{equation*}
1+\frac{zw_{\nu }^{\prime \prime }(z)}{w_{\nu }^{\prime }(z)}=\frac{1}{2}%
\left[ 1-\nu +\sqrt{r}\frac{(2-\nu )\mathbf{H}_{\nu }^{\prime }(\sqrt{r})+%
\sqrt{r}\mathbf{H}_{\nu }^{\prime \prime }(\sqrt{r})}{(1-\nu )\mathbf{H}%
_{\nu }(\sqrt{r})+\sqrt{r}\mathbf{H}_{\nu }^{\prime }(\sqrt{r})}\right]
=1-\sum_{n\geq 1}\frac{z}{\sigma _{\nu ,n}^{2}-z}.
\end{equation*}%
By using inequality (\ref{4}), for all $z\in U(\sigma _{\nu ,1})$ we obtain
the inequality 
\begin{equation}
\Re \left( 1+\frac{zw_{\nu }^{\prime \prime }(z)}{w_{\nu }^{\prime }(z)}%
\right) \geq 1-\sum_{n\geq 1}\frac{r}{\sigma _{\nu ,n}^{2}-r}  \label{eq226}
\end{equation}%
where $\left\vert z\right\vert =r.$\newline
Moreover, again by using inequality (\ref{4}), for all $z\in U(\sigma _{\nu
,1})\text{ and }\beta \geq 0$ we get the inequality 
\begin{eqnarray}
\beta \left\vert \frac{zw_{\nu }^{\prime \prime }(z)}{w_{\nu }^{\prime }(z)}%
\right\vert  &=&\beta \left\vert \sum_{n\geq 1}\frac{z}{\sigma _{\nu
,n}^{2}-z}\right\vert   \label{eq227} \\
&\leq &\beta \sum_{n\geq 1}\left\vert \frac{z}{\sigma _{\nu ,n}^{2}-z}%
\right\vert   \notag \\
&\leq &\beta \sum_{n\geq 1}\frac{r}{\sigma _{\nu ,n}^{2}-r}=-\beta \frac{%
rw_{\nu }^{\prime \prime }(r)}{w_{\nu }^{\prime }(r)}.  \notag
\end{eqnarray}%
As a result, the following inequality be infered from (\ref{eq226}) and (\ref%
{eq227}) 
\begin{equation}
\Re \left( 1+\frac{zw_{\nu }^{\prime \prime }(z)}{w_{\nu }^{\prime }(z)}%
\right) -\beta \left\vert \frac{zw_{\nu }^{\prime \prime }(z)}{w_{\nu
}^{\prime }(z)}\right\vert -\alpha \geq 1-\alpha +(1+\beta )\frac{rw_{\nu
}^{\prime \prime }(r)}{w_{\nu }^{\prime }(r)},~~\beta \geq 0\text{ and }%
\alpha \in \lbrack 0,1).  \label{eq228}
\end{equation}%
where $\left\vert z\right\vert =r.$ So, for $r\in (0,\sigma _{\nu
,1}),~\beta \geq 0\text{ and }\alpha \in \lbrack 0,1)$ we have 
\begin{equation*}
\inf_{\left\vert z\right\vert <r}\left[ \Re \left( 1+\frac{zw_{\nu }^{\prime
\prime }(z)}{w_{\nu }^{\prime }(z)}\right) -\beta \left\vert \frac{zw_{\nu
}^{\prime \prime }(z)}{w_{\nu }^{\prime }(z)}\right\vert -\alpha \right]
=1-\alpha +(1+\beta )\frac{rw_{\nu }^{\prime \prime }(r)}{w_{\nu }^{\prime
}(r)}.
\end{equation*}%
The mapping $\Phi _{\nu }:(0,\sigma _{\nu ,1})\rightarrow \mathbb{R}$
defined by 
\begin{equation*}
\Phi _{\nu }(r)=1+(1+\beta )\frac{rw_{\nu }^{\prime \prime }(r)}{w_{\nu
}^{\prime }(r)}=1-(1+\beta )\sum_{n\geq 1}\frac{r}{\sigma _{\nu ,n}^{2}-r}
\end{equation*}%
is strictly decreasing since $\lim_{r\searrow 0}\Phi _{\nu }(r)=1>\alpha $
and $\lim_{r\nearrow \sigma _{\nu ,1}}\Phi _{\nu }(r)=-\infty $.
Consequently, in wiew of the minimum principle for harmonic functions it
follows that for $\alpha \in \lbrack 0,1),~\beta \geq 0$ and $z\in U(r_{5})$
we obtain 
\begin{equation*}
\Re \left( 1+\frac{zw_{\nu }^{\prime \prime }(z)}{w_{\nu }^{\prime }(z)}%
\right) -\beta \left\vert \frac{zw_{\nu }^{\prime \prime }(z)}{w_{\nu
}^{\prime }(z)}\right\vert >\alpha .
\end{equation*}%
if and only if $r_{5}$ is the unique root of 
\begin{equation*}
1+(1+\beta )\frac{rw_{\nu }^{\prime \prime }(r)}{w_{\nu }^{\prime }(r)}%
=\alpha ,~~\alpha \in \lbrack 0,1)\text{ and }\beta \geq 0.
\end{equation*}%
situated in $(0,\sigma _{\nu ,1})$.
\end{proof}

As a result of the Theorem \ref{thrm26}, the next corollary is obtained by taking $\alpha=0$ ve $\beta=1$.

\begin{corollary}\label{snc26}
Let $\left\vert \nu \right\vert \leq \frac{1}{2}$. Then the uniformly convexity
of the function $w_{\nu }$ is the smallest positive root of the equation 
\begin{equation*}
-\nu +\sqrt{r}\frac{(2-\nu )\mathbf{H}_{\nu
}^{\prime }(\sqrt{r})+\sqrt{r}\mathbf{H}_{\nu }^{\prime \prime }(\sqrt{r})}{%
(1-\nu )\mathbf{H}_{\nu }(\sqrt{r})+\sqrt{r}\mathbf{H}_{\nu }^{\prime }(%
\sqrt{r})}=0.
\end{equation*}%
Moreover $r_{w_{\nu }}^{uc}<r_{w_{\nu }}^{c}<\sigma _{\nu
,1}<h_{\nu ,1},$ where $h_{\nu ,1}$ and $\sigma _{\nu ,1}$ denote the first
positive zeros of $\mathbf{H}_{\nu }$ and $w_{\nu }^{\prime },$ respectively.
\end{corollary}

\begin{center}
\includegraphics{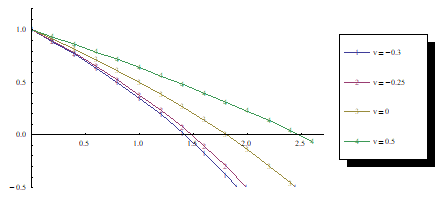}

The graph of the fuction $r\mapsto  -\nu +\sqrt{r}\frac{(2-\nu )\mathbf{H}_{\nu
}^{\prime }(\sqrt{r})+\sqrt{r}\mathbf{H}_{\nu }^{\prime \prime }(\sqrt{r})}{%
(1-\nu )\mathbf{H}_{\nu }(\sqrt{r})+\sqrt{r}\mathbf{H}_{\nu }^{\prime }(%
\sqrt{r})}$  \vskip 0.3 cm

for $\nu \in \{-0.3, -0.25, 0, 0.5\}$ on $[0,2.7]$
\end{center}

Using the following representation of Struve functions of order $1/2$ in terms of elementary trigonometric functions
$$\mathbf{H}_{\frac{1}{2}}(z)= \sqrt{\frac{2}{\pi z}} \left(1-\cos z\right) $$
we obtain 
$$u_{\frac{1}{2}}(z)= 2^{\frac{2}{3}}{\left( \frac{1-\cos z}{\sqrt{z}} \right)}^{\frac{2}{3}},~~~ v_{\frac{1}{2}}(z)=2\left(\frac{1-\cos z}{z} \right) \text{ and } w_{\frac{1}{2}}(z)=2\left(1-\cos \sqrt{z} \right). $$
We state the following results for the functions $u_{\frac{1}{2}},~v_{\frac{1}{2}}$ and $w_{\frac{1}{2}} $.
\begin{itemize}
\item $ u_{\frac{1}{2}}(z) \in UC$ in the disk $U(r_1=1.1382)$,
\item $ v_{\frac{1}{2}}(z) \in UC$ in the disk $U(r_2=0.9349)$,
\item $ w_{\frac{1}{2}}(z) \in UC$ in the disk $U(r_3=2.4674)$,
\end{itemize}
where $r_1,~r_2 $ and $r_3$ is the smallest positive root of the equations \begin{itemize}
\item $5+\left(-5+8z^2\right)\cos z-2z\left(2z+\sin z\right)=0$,
\item $\cos z\left(2z^2-3\right)-z\sin z+3=0$,
\item $\sqrt{z}\cot \sqrt{z}=0 $, respectively.
\end{itemize}

\begin{remark}\label{yrm1}
For $\beta=0$, Theorems \ref{thrm21}, \ref{thrm22} and \ref{thrm23} reduce to \cite[Thm. $3$, (a),(b) and (c)]{Baricz5}, respectively, for the case $\mu \in (-1,1),~ \mu \neq 0$ and $\mu \neq -\frac{1}{2}$. \\ 
Moreover, Theorems \ref{thrm24}, \ref{thrm25} and \ref{thrm26} reduce to \cite[Thm. $4$, (a),(b) and (c)]{Baricz5}, respectively, on putting $\beta=0$, for the case $|\nu| \leq \frac{1}{2}$. 

\end{remark}


\begin{thebibliography}{99}
\bibitem{Baricz02} Á. Baricz, P. A. Kupán, R. Szász, The radius of starlikeness
of normalized Bessel functions of the first kind, Proc. Amer. Math. Soc.
142(5) (2014) 2019--2025.

\bibitem{Baricz03} Á. Baricz, R. Szász, The radius of convexity of normalized
Bessel functions of the first kind, Anal. Appl. 12(5) (2014) 485-509.

\bibitem{Baricz04} Á. Baricz, H. Orhan, R. Szász, The radius of $\alpha -$%
convexity of normalized Bessel functions of the first kind, Comput. Method.
Func. Theo. 16(1) (2016) 93-103.

\bibitem{Baricz05} Á. Baricz, R. Szász, The radius of convexity of normalized
Bessel functions, Analysis Math. 41(3) (2015) 141-151.

\bibitem{Baricz06} Á. Baricz, M. Ça\u{g}lar, E. Deniz, Starlikeness of Bessel
functions and their derivatives, Math. Ineq. Appl. 19(2) (2016) 439--449.

\bibitem{Baricz1} Á. Baricz, S. Koumandos, Turán type inequalities for some
Lommel functions of the first kind. Proc. Edinb. Math. Soc. 59(3) (2016),
569-579.

\bibitem{Baricz2} Á. Baricz, R. Szász, Close-to-convexity of some special
functions and their derivatives. Bull. Malays. Math. Sci. Soc. 39(1) (2016),
427-437.

\bibitem{Baricz3} Á. Baricz, D. K. Dimitrov, H. Orhan, N. Ya\u{g}mur, Radii
of starlikeness of some special functions. Proc. Amer. Math. Soc. 144
(2016), 3355-3367.

\bibitem{Baricz4} Á. Baricz, S. Ponnusamy, S. Singh, Turán type inequalities
for Struve functions. J. Math. Anal. Appl.
445(1) (2017), 971-984.

\bibitem{Baricz5} Á. Baricz, N. Ya\u gmur, Geometric properties of some
Lommel and Struve functions. Ramanujan J. 42 (2017), 325-346.


\bibitem{Bhar} R. Bharti, R. Parvatham and A. Swaminathan, On subclasses
of uniformly convex functions and corresponding class of starlike functions.
Tamkang J. Math. 28 (1997), 17-32.

\bibitem{Deniz} E. Deniz, R. Szász, The radius of uniform convexity of
Bessel functions, J. Math. Anal. Appl., 453(1)
2017, 572-588.

\bibitem{Goo} A. W. Goodman, On uniformly convex functions, Ann. Polon.
Math. 56 (1991) 87-92.

\bibitem{Kanas1} S. Kanas, A. Wisniowska, Conic regions and k-uniform convexity, J. Comput. Appl. Math. 105 (1999) 327–336.

\bibitem{Kanas2} S. Kanas, A. Wisniowska, Conic domains and starlike functions, Rev. Roumaine Math. Pures Appl. 45 (2000) 647–657.

\bibitem{Kanas3} S. Kanas, T. Yaguchi, Subclasses of k-uniform convex and starlike functions defined by generalized derivative, I, Indian J. Pure Appl. Math.  32 (9) (2001) 1275-1282.

\bibitem{Koumandos} S. Koumandos, M. Lamprecht, The zeros of certain Lommel
functions. Proc. Amer. Math. Soc. 140, 3091-3100 (2012).

\bibitem{Levin} B. Ya. Levin, Lectures on entire functions. Vol. 150. Amer.
Math. Soc. (1996).

\bibitem{Orhan} H. Orhan, N. Ya\u{g}mur, Geometric properties of generalized
Struve functions. An. Ştiint. Univ. Al. I.Cuza Ia şi. Mat. (N.S) (2014).

\bibitem{Rav} V. Ravichandran, On uniformly convex functions, Ganita 53(2)
(2002) 117--124.

\bibitem{10} F. R\o nning, Uniformly convex functions and a corresponding
class of starlike functions, Proc. Amer. Math. Soc. 118(1) (1993) 189-196.

\bibitem{Steinig70} J. Steinig, The real zeros of Struve's functions. SIAM
J. Math. Anal. \textbf{1(}3\textbf{), }365-375 (1970).

\bibitem{Steinig72} J. Steinig, The sign of Lommel's Function. Trans. Am.
Math. Soc. \textbf{163}, 123-129 (1972).

\bibitem{watson} G. N. Watson, A Treatise of the Theory of Bessel Functions.
Cambridge University Press, Cambridge, 1944.

\bibitem{yagmur} N. Ya\u{g}mur, H. Orhan, Starlikeness and convexity of
generalized Struve functions. Abstr. Appl. Anal. 2013: Art. 954513 (2013).
\end{thebibliography}
\end{document}